\RequirePackage{iftex}
\ifpTeX
  \documentclass[a4paper,dvipdfmx,11pt]{amsart}
\else
  \documentclass[a4paper,11pt]{amsart}
\fi
\usepackage{amsmath,amsthm,amssymb,mathrsfs,stmaryrd,mathtools}
\allowdisplaybreaks
\usepackage{bm}
\usepackage{url}
\usepackage{color}
\usepackage{eucal}
\usepackage{slashed}
\usepackage{physics}
\usepackage{graphicx}
\usepackage[alpine]{ifsym}
\usepackage{fontawesome5}
\usepackage{tikz}
\usetikzlibrary{nfold,arrows}
\usepackage{tikz-cd}
\usetikzlibrary{intersections, calc, arrows, arrows.meta, shadows}
\usepackage[utf8]{inputenc} 
\usepackage[T1]{fontenc}
\usepackage[style=alphabetic,sorting=nty]{biblatex}
\usepackage{geometry} 
\usepackage[inline]{enumitem} 
\setlist[enumerate]{itemsep=2pt,parsep=2pt,before={\parskip=2pt}}
\setlist[description]{style=standard}
\usepackage[colorlinks=true, hyperindex, linkcolor=magenta, pagebackref=false, citecolor=cyan, pdfpagelabels]{hyperref} 
\usepackage{aliascnt}
\usepackage{cleveref}
\usepackage{quiver}
\usepackage{mathtools}
\mathtoolsset{showonlyrefs=true}

\newcommand{\calB}{\mathcal{B}}
\newcommand{\calC}{\mathcal{C}}
\newcommand{\calD}{\mathcal{D}}
\newcommand{\calE}{\mathcal{E}}

\newcommand{\calU}{\mathcal{U}}

\newcommand{\rmN}{\mathrm{N}}

\newcommand{\rmu}{\mathrm{u}}

\newcommand{\bbI}{\mathbb{I}}

\newcommand{\C}{\calC}
\newcommand{\D}{\calD}
\newcommand{\E}{\calE}

\DeclareMathOperator{\Const}{const}

\newcommand{\An}{\mathrm{An}}

\newcommand{\Cat}{\mathrm{Cat}}

\newcommand{\CSeg}{\mathrm{CSeg}}

\newcommand{\Fun}{\mathrm{Fun}}

\newcommand{\PrL}{\mathrm{Pr}^{L}}
\newcommand{\PrR}{\mathrm{Pr}^{R}}

\newcommand{\Seg}{\mathrm{Seg}}
\newcommand{\Set}{\mathrm{Set}}

\newcommand{\sSpace}{\mathrm{sSpace}}

\newcommand{\sSet}{\mathrm{sSet}}

\newcommand{\ac}{\mathrm{ac}}

\newcommand{\const}{\mathrm{const}}

\newcommand{\co}{\mathrm{co}}

\DeclareMathOperator{\colim}{colim}

\renewcommand{\ev}{\mathrm{ev}} 
\newcommand{\fib}{\mathrm{fib}}

\DeclareMathOperator{\Hom}{Hom}
\newcommand{\inc}{\mathrm{inc}}
\newcommand{\inj}{\mathrm{inj}}
\DeclareMathOperator{\id}{id}

\newcommand{\inv}{\mathrm{inv}}

\newcommand{\KQ}{\mathrm{KQ}}
\newcommand{\Lan}{\mathrm{Lan}}
 
\DeclareMathOperator{\Map}{Map}

\DeclareMathOperator{\myop}{op}

\DeclareMathOperator{\pr}{pr} 

\newcommand{\rev}{\mathrm{rev}} 

\newcommand{\str}{\mathrm{str}} 
\newcommand{\Sp}{\mathrm{Sp}} 
\newcommand{\triv}{\mathrm{triv}} 

\newcommand{\flag}{\mathord{\text{\faFlag}}}
\newcommand{\AI}{\mathrm{AI}}
\newcommand{\AInvCat}{\mathrm{AICat}}
\newcommand{\Herm}{\mathrm{Herm}}
\newcommand{\HermFlag}{\Herm^{\flag}}
\newcommand{\image}{\operatorname{im}}

\newcommand{\FlDagCat}{\Cat^{\flag,\dag}}
\newcommand{\coFlDagCat}{\Cat^{\co\flag,\dag}}
\newcommand{\DagCatInf}{\Cat^{\dag}}

\newcommand{\SegRev}{\Seg_{\rev}}
\newcommand{\SegRevU}{\Seg_{\rev}^{u}}

\newcommand{\flagtag}{\mathrm{flag}}
\newcommand{\eqtag}{\mathrm{eq}}
\newcommand{\conj}{\mathrm{conj}}

\newcommand{\fftag}{\mathrm{ff}}
\newcommand{\pstag}{\mathrm{ps}}

\newcommand{\estag}{\mathrm{es}}

\newcommand{\sfK}{\mathsf{K}}

\newcommand{\oneDag}{\mathbf{1}_{\dagger}}

\usepackage{relsize}
\usepackage[bbgreekl]{mathbbol}
\usepackage{amsfonts}
\DeclareSymbolFontAlphabet{\mathbb}{AMSb} 
\DeclareSymbolFontAlphabet{\mathbbl}{bbold}
\newcommand{\prism}{{\mathlarger{\mathbbl{\Delta}}}}

\DeclareFontFamily{U}{dmjhira}{}
\DeclareFontShape{U}{dmjhira}{m}{n}{
  <-> dmjhira
}{}
\DeclareFontSubstitution{U}{dmjhira}{m}{n}
\newcommand{\yo}{\text{{\usefont{U}{dmjhira}{m}{n}\symbol{"48}}}}

\newcommand{\xtwoheadrightarrow}[1]{%
  \overset{#1}{\twoheadrightarrow}%
}

\theoremstyle{definition}
\newtheorem{theorem}{Theorem}[subsection]

\newaliascnt{conjecture}{theorem}

\aliascntresetthe{conjecture}
\crefname{conjecture}{conjecture}{conjectures}
\Crefname{conjecture}{Conjecture}{Conjectures}

\newaliascnt{construction}{theorem}
\newtheorem{construction}[construction]{Construction}
\aliascntresetthe{construction}
\crefname{construction}{construction}{constructions}
\Crefname{construction}{Construction}{Constructions}

\newaliascnt{convention}{theorem}
\newtheorem{convention}[convention]{Convention}
\aliascntresetthe{convention}
\crefname{convention}{convention}{conventions}
\Crefname{convention}{Convention}{Conventions}

\newaliascnt{corollary}{theorem}
\newtheorem{corollary}[corollary]{Corollary}
\aliascntresetthe{corollary}
\crefname{corollary}{corollary}{corollaries}
\Crefname{corollary}{Corollary}{Corollaries}

\newaliascnt{definition}{theorem}
\newtheorem{definition}[definition]{Definition}
\aliascntresetthe{definition}
\crefname{definition}{definition}{definitions}
\Crefname{definition}{Definition}{Definitions}

\newaliascnt{example}{theorem}

\aliascntresetthe{example}
\crefname{example}{example}{examples}
\Crefname{example}{Example}{Examples}

\newaliascnt{lemma}{theorem}
\newtheorem{lemma}[lemma]{Lemma}
\aliascntresetthe{lemma}
\crefname{lemma}{lemma}{lemmas}
\Crefname{lemma}{Lemma}{Lemmas}

\newaliascnt{notation}{theorem}
\newtheorem{notation}[notation]{Notation}
\aliascntresetthe{notation}
\crefname{notation}{notation}{notations}
\Crefname{notation}{Notation}{Notations}

\newaliascnt{observation}{theorem}

\aliascntresetthe{observation}
\crefname{observation}{observation}{observations}
\Crefname{observation}{Observation}{Observations}

\newaliascnt{proposition}{theorem}
\newtheorem{proposition}[proposition]{Proposition}
\aliascntresetthe{proposition}
\crefname{proposition}{proposition}{propositions}
\Crefname{proposition}{Proposition}{Propositions}

\newaliascnt{remark}{theorem}
\newtheorem{remark}[remark]{Remark}
\aliascntresetthe{remark}
\crefname{remark}{remark}{remarks}
\Crefname{remark}{Remark}{Remarks}

\newaliascnt{fact}{theorem}

\aliascntresetthe{fact}
\crefname{fact}{fact}{facts}
\Crefname{fact}{Fact}{Facts}

\newaliascnt{terminology}{theorem}

\aliascntresetthe{terminology}
\crefname{terminology}{terminology}{terminologies}
\Crefname{terminology}{Terminology}{Terminologies}

\theoremstyle{definition}
\newtheorem*{theorem*}{Theorem}
\newtheorem*{conjecture*}{Conjecture}
\newtheorem*{construction*}{Construction}
\newtheorem*{convention*}{Convention}
\newtheorem*{corollary*}{Corollary}
\newtheorem*{definition*}{Definition}
\newtheorem*{example*}{Example}
\newtheorem*{lemma*}{Lemma}
\newtheorem*{notation*}{Notation}
\newtheorem*{observation*}{Observation}
\newtheorem*{proposition*}{Proposition}
\newtheorem*{remark*}{Remark}
\newtheorem*{terminology*}{Terminology}

\newtheorem*{maintheorema}{Theorem A}
\newtheorem*{maintheoremb}{Theorem B}
\newtheorem*{maintheoremc}{Theorem C}
\newtheorem*{maintheoremd}{Theorem D}

\title[Models for dagger $(\infty,1)$-categories III]{Models for dagger $(\infty,1)$-categories III:\\Univalent flagged dagger $(\infty,1)$-categories and unitarily complete dagger Segal anima}

\author{Keima Akasaka} 

\thanks{Graduate School of Science and Engineering, Chiba University.
  Email: \href{mailto:quasi.cosmoi@gmail.com}{\texttt{quasi.cosmoi@gmail.com}}.
}

\date{\today}

\begin{document}

\begin{abstract}
  We identify the $(\infty,1)$-category presented by the dagger Bergner, dagger Joyal, and dagger Rezk models of Parts I and II with that of univalent flagged dagger $(\infty,1)$-categories.
  We first construct a realization--nerve adjunction between reversal simplicial anima and Hermitian flags.
  We prove that its nerve restricts to an equivalence between flagged dagger $(\infty,1)$-categories and dagger Segal anima, under which univalence corresponds to unitary completeness.
  We then construct the adjunctions for univalentization and Hermitian completion and identify univalentization with unitary completion.
\end{abstract}

\maketitle
\setcounter{tocdepth}{1}
\tableofcontents

\section{Introduction}\label{intro.section}

In Parts~I \cite{Aka26} and~II \cite{Aka26-2}, we constructed the dagger Bergner, dagger Joyal, and dagger Rezk model structures and related them by Quillen equivalences.
The purpose of this third part is to identify their common underlying $(\infty,1)$-category with the $(\infty,1)$-category of univalent flagged dagger $(\infty,1)$-categories in the sense of \cite{FerrerEtAl24}.
We establish this comparison through a realization--nerve adjunction between reversal simplicial anima and Hermitian flags.

The comparison must account for both the underlying anti-involution and the chosen Hermitian structures on objects.
The realization of a dagger Segal anima yields a Hermitian flag whose underlying flag is essentially surjective.
Unitary completeness then requires this flag to be fully faithful as a morphism into the Hermitian anima.
Separating these two conditions allows us to construct the nerve on all Hermitian flags and to identify precisely the subcategories on which it is an equivalence.

\subsection*{Definitions}

We briefly describe the objects which occur in the comparison.
An \emph{anti-involutive $(\infty,1)$-category} is a homotopy fixed point of the opposite action on $\Cat_{(\infty,1)}$.
For such an object $\calC$, write $\calC_*$ for its underlying $(\infty,1)$-category.
The homotopy fixed points of its equivariant core form the anima $\Herm(\calC)$ of \emph{Hermitian structures}, with a forgetful morphism $\epsilon_{\calC}:\Herm(\calC)\to\iota_0\calC_*$.

A \emph{Hermitian flag} is a morphism $i:A\to\Herm(\calC)$, with an anima $A$ and an anti-involutive $(\infty,1)$-category $\calC$.
Its \emph{underlying flag} is the composite $\epsilon_{\calC}i:A\to\iota_0\calC_*$.
A Hermitian flag is called
\begin{itemize}
  \item a \emph{flagged dagger $(\infty,1)$-category} if $\epsilon_{\calC}i$ is essentially surjective;
  \item a \emph{coflagged dagger $(\infty,1)$-category} if $i$ is fully faithful;
  \item a \emph{univalent flagged dagger $(\infty,1)$-category} (or a \emph{dagger $(\infty,1)$-category}) if both conditions hold.
\end{itemize}
We denote the full subcategories of flagged, coflagged, and univalent flagged dagger $(\infty,1)$-categories by $\FlDagCat_{(\infty,1)}$, $\coFlDagCat_{(\infty,1)}$, and $\DagCatInf_{(\infty,1)}$, respectively.

\subsection*{The main theorems}

Let $\prism_{\rev}$ be the reversal simplex category.
We write $\SegRev(\An)\subseteq\Fun(\prism_{\rev}^{\myop},\An)$ for the full subcategory of presheaves whose restriction to $\prism^{\myop}$ satisfies the Segal condition.
Unitary completeness is locality with respect to $e_{\rev}:\yo_{\rev}[0]\to E_{\rev}$ (\cref{direct.def.unitary-complete}).
Here $E_{\rev}([n])=\{0,1\}^{n+1}$, with morphisms induced by precomposition on vertices, and $e_{\rev}$ selects the constant tuple $0$.
The full subcategory of unitarily complete dagger Segal anima is denoted by $\SegRevU(\An)$.

The main construction is the realization--nerve adjunction
\begin{align}
  \ac_{\flag}:\Fun(\prism_{\rev}^{\myop},\An)
  \rightleftarrows\HermFlag_{(\infty,1)}:\rmN_{\flag}
\end{align}
obtained from the walking flagged dagger simplices of \cref{direct.construction.nerve}.
Every value of the ambient nerve is Segal.
Computing the unit and counit shows that the nerve is fully faithful on flagged objects, and that every dagger Segal anima occurs in this way (\cref{direct.thm.unit,direct.thm.counit}).
The walking unitary then detects the coflagged condition (\cref{direct.thm.unitary}), giving our first main result.

\begin{maintheorema}[\cref{direct.thm.flagged-equivalence,direct.cor.unitary-equivalence}]
  The flagged nerve $\rmN_{\flag} : \HermFlag_{(\infty,1)} \to \Fun(\prism_{\rev}^{\myop},\An)$ restricts to equivalences of $(\infty,1)$-categories
  \begin{align}
    \rmN_{\flag}|_{\FlDagCat_{(\infty,1)}}:
    \FlDagCat_{(\infty,1)}
    \xrightarrow{\simeq}\SegRev(\An)
    \quad \text{and} \quad
    \rmN_{\flag}|_{\DagCatInf_{(\infty,1)}}:
    \DagCatInf_{(\infty,1)}
    \xrightarrow{\simeq}\SegRevU(\An).
  \end{align}
\end{maintheorema}

We next compare the completeness morphism with the point-set generator used in Part~II.
Together with the spine morphisms, it defines the localization set $S_{\rev}$ of \cref{presentation.rem.localization-locality}.
Passing from simplicial sets to anima identifies this localization with the unitarily complete objects in Theorem~A.

\begin{maintheoremb}[\cref{presentation.prop.reduction}]
  There is an equivalence of $(\infty,1)$-categories
  \begin{align}
    (L_{S_{\rev}}\Fun(\prism_{\rev}^{\myop},\sSet)_{\inj})_{(\infty,1)}
    \simeq
    \DagCatInf_{(\infty,1)}.
  \end{align}
\end{maintheoremb}

The Quillen equivalences of Parts~I and~II therefore show that the dagger Bergner, dagger Joyal, and dagger Rezk model categories all present $\DagCatInf_{(\infty,1)}$ (\cref{presentation.cor.all-models}).

The two conditions on a Hermitian flag also determine two adjunctions.
Image factorization replaces $i:A\to\Herm(\calC)$ by the inclusion of its image and defines a reflection onto coflagged objects (\cref{flag.prop.ambient-coflagged-reflection}).
Its restriction to flagged objects is univalentization.

\begin{maintheoremc}[\cref{flag.thm.univalentization,flag.cor.unitary-completion}]
  The inclusion $\DagCatInf_{(\infty,1)} \hookrightarrow \FlDagCat_{(\infty,1)}$ admits a left adjoint
  \begin{align}
    L_{\co\flag}|_{\flag} : \FlDagCat_{(\infty,1)} \to \DagCatInf_{(\infty,1)}.
  \end{align}
  Under the equivalence of Theorem~A, this adjunction corresponds to the reflective localization $\SegRev(\An) \to \SegRevU(\An)$ with respect to the completeness morphism.
\end{maintheoremc}

Restricting the underlying anti-involutive $(\infty,1)$-category to the full subcategory spanned by objects equivalent to those selected by the flag gives a coreflection onto flagged objects (\cref{flag.prop.ambient-flagged-coreflection}).
On coflagged objects, this gives Hermitian completion.

\begin{maintheoremd}[\cref{flag.thm.hermitian-completion}]
  The inclusion $\DagCatInf_{(\infty,1)} \hookrightarrow \coFlDagCat_{(\infty,1)}$ admits a right adjoint
  \begin{align}
    R_{\flag}|_{\co\flag} : \coFlDagCat_{(\infty,1)} \to \DagCatInf_{(\infty,1)}.
  \end{align}
\end{maintheoremd}

The ambient reflection and coreflection commute in the sense of \cref{properties.prop.commuting-corrections}.
We also describe unitary equivalences by relative homotopy fixed points and show that $E_{\rev}$ represents their anima (\cref{direct.thm.coherent-unitary-representability}).
This description gives a Dwyer--Kan criterion for equivalences after unitary completion (\cref{presentation.thm.unitary-dk}).

\subsection*{Organization}

This third part is organized as follows.
\begin{itemize}
  \item \Cref{direct.section.targets} recalls homotopy fixed points and defines anti-involutive $(\infty,1)$-categories, Hermitian flags, and the flagged and coflagged conditions;
  \item \Cref{direct.section} constructs the realization--nerve adjunction and computes its unit and counit.
  It then characterizes unitary completeness and proves the two equivalences of Theorem~A;
  \item \Cref{flag.section} constructs the two flag adjunctions of Theorems~C and~D, identifies univalentization with unitary completion, and proves their compatibility;
  \item \Cref{presentation.section} compares the completeness morphism with its point-set counterpart and identifies the model-categorical presentations of Theorem~B;
  \item \Cref{unitary.section} describes unitary equivalences by relative homotopy fixed points and proves the unitary Dwyer--Kan criterion.
\end{itemize}

\subsection*{Notation}

We fix some notational conventions.
\begin{itemize}
  \item Fix universes $\mathcal U\subset\mathcal V\subset\mathcal W$.
  \item The symbols $\An$ and $\Cat_{(\infty,1)}$ refer to $\mathcal U$-small anima and $(\infty,1)$-categories;
  their hatted versions $\widehat{\An}$ and $\widehat{\Cat}_{(\infty,1)}$ refer to $\mathcal V$-small objects.
  \item The notation $\Map_{\calE}(x,y)$ denotes an anima, whereas $\Fun(\calC,\calD)$ denotes the entire functor $(\infty,1)$-category.
  In particular, $\Map_{\Cat_{(\infty,1)}}(\calC,\calD)=\iota_0\Fun(\calC,\calD)$.
  \item Adjunctions between functors, and their equivariant units and counits, are understood in the standard $(\infty,2)$-category of $(\infty,1)$-categories, functors, and natural transformations.
\end{itemize}

\subsection*{AI declaration}

In preparing this paper, the author used GPT-6 Astra to improve the English and identify linguistic errors. 
It was also used to identify relevant references, allowing the author to shorten independently developed proofs by citing existing results. 
The author assumes full responsibility for all mathematical content.

\subsection*{Acknowledgements}

This work was supported by JST SPRING, Grant Number JPMJSP2109.

\section{Univalent flagged dagger \texorpdfstring{$(\infty,1)$}{(∞,1)}-categories}\label{direct.section.targets}\label{target.section}

In this section, we recall the definitions of univalent flagged dagger $(\infty,1)$-categories in the sense of \cite{FerrerEtAl24}.
We first recall homotopy fixed points and the equivariant core, and then organize choices of Hermitian structures into the $(\infty,1)$-category of Hermitian flags (\cref{target.def.hermitian-flags}).

\subsection{\texorpdfstring{$G$}{G}-actions}

We recall group actions, homotopy fixed points, and the equivariant adjunctions used below (\cref{target.lem.homotopy-fixed-point-formalism}).
Let $G$ be a small group and let $\calE$ be an $(\infty,1)$-category.
Although we only need $G=C_2$, the statements in this subsection apply to general $G$.

\begin{definition}
  A \emph{$G$-action} in $\calE$ is a functor $X : BG \to \calE$.
  A \emph{morphism of $G$-actions} $F : X \to Y$ is a natural transformation from $X$ to $Y$.
  We let $\Fun(BG,\calE)$ denote the $(\infty,1)$-category of $G$-actions and morphisms of $G$-actions in $\calE$.
\end{definition}

\begin{convention}
  When $\E$ is $\Cat_{(\infty,1)}$ or $\widehat{\Cat}_{(\infty,1)}$, we will call a $G$-action a \emph{$G$-category}, and a morphism of $G$-actions a \emph{$G$-equivariant functor}.

  Similarly, when $\E$ is $\An$ or $\widehat{\An}$, we will call a $G$-action a \emph{$G$-anima}, and a morphism of $G$-actions a \emph{$G$-equivariant morphism}.
\end{convention}

\begin{definition}
  Let $X\in\Fun(BG,\calE)$.
  Define the \emph{homotopy fixed-point object} $X^{hG}$ of $X$, if it exists, by the $BG$-shaped limit in $\calE$:
  \begin{align}
    X^{hG}
    :=
    \lim_{BG} X.
  \end{align}
\end{definition}

\begin{remark} \label{target.rem.constant-limit-adjunction}
  The unique functor $BG\to *$ induces a functor
  \begin{align}
    \triv:
    \calE \to \Fun(BG,\calE)
    : X \mapsto (X,\triv) := \Const_{BG}X.
  \end{align}
  If $\calE$ admits $BG$-shaped limits, then the constant-diagram functor admits the homotopy-fixed-point functor as a right adjoint:
  \begin{align}
    \triv:
    \calE
    \rightleftarrows
    \Fun(BG,\calE):
    (-)^{hG}.
  \end{align}
\end{remark}

\begin{notation}
  For a morphism of $G$-actions $F:X \to Y$, write
  \begin{align}
    F_{*}
    : X_{*} := X(*)
    \to Y(*) =: Y_{*}
    \quad \text{and} \quad
    F^{hG}
    :=
    \lim_{BG}F:
    X^{hG}
    \to
    Y^{hG}.
  \end{align}
\end{notation}

\begin{lemma}\label{target.lem.constant-groupoid-limit}
  Let $\calE$ be a locally small $(\infty,1)$-category.
  There is a natural equivalence of $(\infty,1)$-categories
  \begin{align}
    \lim_{BG}
    \Const_{BG}(\calE)
    \simeq
    \Fun(BG,\calE).
  \end{align}
\end{lemma}

\begin{proof}
  The constant diagram $\Const_{BG} \calE$ is classified by the projection $BG\times\calE\to BG$.
  Its limit is computed as the $(\infty,1)$-category of cocartesian sections.
  Since $BG$ is a groupoid, every section sends its arrows to equivalences in $BG\times\calE$, which are cocartesian.
  Hence we have 
  \begin{align}
    \lim_{BG}\Const_{BG}(\calE)
    &\simeq
    \Fun_{/BG}(BG,BG\times\calE)\\
    &\simeq
    \Fun(BG,\calE).
  \end{align}
\end{proof}

\begin{construction}
  Let $\calC:BG\to\Cat_{(\infty,1)}$ and $\calD:BG\to\widehat{\Cat}_{(\infty,1)}$ be $G$-categories.
  Define a $G$-category $\Fun^{\conj}(\calC,\calD):BG\to\widehat\Cat_{(\infty,1)}$ by the composite
  \begin{align}
    BG
    \xrightarrow{(\inv,\id)}
    BG^{\myop}\times BG 
    \xrightarrow{\calC^{\myop}\times\calD}
    \Cat_{(\infty,1)}^{\myop}\times\widehat{\Cat}_{(\infty,1)} 
    \xrightarrow{\Fun(-,-)}
    \widehat{\Cat}_{(\infty,1)}.
  \end{align}
  Its value at $* \in BG$ is $\Fun(\calC_{*},\calD_{*})$.
  Here $\inv: BG \to BG^{\myop} : g \mapsto (g^{-1})^{\myop}$ is inversion, and $\calC^{\myop}:BG^{\myop}\to\Cat_{(\infty,1)}^{\myop}$ is the opposite functor of $\calC$.
\end{construction}

\begin{construction}
  Let $\calC:BG\to\widehat{\Cat}_{(\infty,1)}$ be a $G$-category, and let $x,y\in\calC^{hG}$.
  Let $\widetilde x,\widetilde y:\triv(*)\to\calC$ be their adjuncts under the adjunction $\triv\dashv(-)^{hG}$.
  Define a $G$-anima $\Map_{\calC}(x,y):BG\to\widehat{\An}$ by the following pullback in $\Fun(BG,\widehat{\Cat}_{(\infty,1)})$:
  \begin{align}
    \Map_{\calC}(x,y)
    :=
    \triv(*)
    \times_{(\widetilde x,\widetilde y),\calC\times\calC,(\ev_0,\ev_1)}
    \Fun^{\conj}(\triv([1]),\calC).
  \end{align}
  Its value at $* \in BG$ is $\Map_{\calC_{*}}(x_{*},y_{*})$.
\end{construction}

\begin{lemma}\label{target.lem.homotopy-fixed-point-formalism}
  Let $\calC,\calD:BG\to\widehat{\Cat}_{(\infty,1)}$ be $G$-categories, and let $F : \calC \to \calD$ be an equivariant functor.
  Then the following statements hold.
  \begin{enumerate}
    \item If $F_{*}$ is an equivalence, then $F^{hG}$ is an equivalence.
    \item For $x,y\in\calC^{hG}$, there is a natural equivalence of anima
    \begin{align}
      \Map_{\calC^{hG}}(x,y) 
      \simeq 
      \lim_{BG}\Map_{\calC}(x,y).
    \end{align}
    \item If $F_{*}$ is fully faithful, then $F^{hG}$ is fully faithful.
    \item Suppose that the equivariant functors $L:\calC\to\calD, R:\calD\to\calC$, together with an equivariant unit and counit, form an equivariant adjunction in the sense of \cite[Appendix~B]{Hei26}.
    Then there is an adjunction $L^{hG}\dashv R^{hG}$.
    \item Let $L:\calC\to\calD$ be an equivariant functor.
    If $L_{*}$ admits a right adjoint $R_{*}$, then $R_{*}$ admits an equivariant refinement $R:\calD\to\calC$ for which the unit and counit of $L_{*}\dashv R_{*}$ are equivariant.
    Consequently, there is an adjunction $L^{hG}\dashv R^{hG}$.
    \item Let $\calB \rightarrow \calD \leftarrow \calC$ be a cospan in $\Fun(BG,\widehat{\Cat}_{(\infty,1)})$.
    Form the pullback $\calB \times_{\calD} \calC$ in $\Fun(BG,\widehat{\Cat}_{(\infty,1)})$.
    Then there are natural equivalences of $(\infty,1)$-categories
    \begin{align}
      (\calB\times_{\calD}\calC)_{*} 
      \simeq \calB_{*} \times_{\calD_{*}} \calC_{*}
      \quad \text{and} \quad 
      (\calB\times_{\calD}\calC)^{hG}
      \simeq
      \calB^{hG} \times_{\calD^{hG}} \calC^{hG}.
    \end{align}
    \item Let $\calC \in \Cat_{(\infty,1)}$.
    There is an equivalence of $(\infty,1)$-categories
    \begin{align}
      \Fun^{\conj}((\calC,\triv),\calD)^{hG}
      \simeq
      \Fun(\calC,\calD^{hG}).
    \end{align}
  \end{enumerate}
\end{lemma}

\begin{proof}
  (1)
  Since $F_*$ is an equivalence, $F$ is an equivalence in $\Fun(BG,\widehat\Cat_{(\infty,1)})$.
  Taking the limit preserves this equivalence.

  (2)
  The functor $\Fun([1],-)$ preserves limits, and the $G$-action on $[1]$ is trivial.
  Since limits commute with pullbacks, we have
  \begin{align}
    \lim_{BG}\Map_{\calC}(x,y)
    &\simeq
    \{(x,y)\}
    \times_{\calC^{hG}\times\calC^{hG}}
    \lim_{BG}\Fun^{\conj}(\triv([1]),\calC) \\
    &\simeq
    \{(x,y)\}
    \times_{\calC^{hG}\times\calC^{hG}}
    \Fun([1],\calC^{hG}) \\
    &\simeq
    \Map_{\calC^{hG}}(x,y).
  \end{align}

  (3)
  The functor $F:\calC\to\calD$ induces a morphism of $G$-anima $\Map_{\calC}(x,y)\to\Map_{\calD}(F^{hG}x,F^{hG}y)$.
  Its value at $* \in BG$ is $\Map_{\calC_{*}}(x_{*},y_{*}) \to \Map_{\calD_{*}}(F_{*}x_{*},F_{*}y_{*})$, which is an equivalence by full faithfulness of $F_*$.
  Taking limits and using (2) proves full faithfulness of $F^{hG}$.

  (4)
  The limit $2$-functor preserves adjunctions as in \cite[proof of Lemma~B.1]{Hei26}.
  Applying it to the given equivariant adjunction gives $L^{hG}\dashv R^{hG}$.

  (5)
  All transition functors are equivalences, so the right-adjoint mate of each equivariance square of $L$ is an equivalence.
  The dual of \cite[Lemma~B.1]{Hei26} supplies the equivariant right adjoint together with an equivariant unit and counit.
  Applying (4) gives the homotopy-fixed-point adjunction.

  (6)
  Let $\sfK$ be the cospan category $\{0\to 1\leftarrow 2\}$.
  The given cospan determines a functor $H: BG\times\sfK \to \widehat{\Cat}_{(\infty,1)}$.
  Taking the $\sfK$-shaped limit in the second variable defines a functor 
  \begin{align}
    (\calB\times_{\calD}\calC)(-)
    :=
    \lim_{\sfK}H(-,-):
    BG
    \to
    \widehat{\Cat}_{(\infty,1)}.
  \end{align}
  Since limits in functor $(\infty,1)$-categories are computed pointwise, we have 
  \begin{align}
    (\calB\times_{\calD}\calC)_{*}
    &\simeq
    \lim_{\sfK}H(*,-)
    \simeq
    \calB_{*} \times_{\calD_{*}} \calC_{*}, \\
    (\calB\times_{\calD}\calC)^{hG}
    &\simeq
    \lim_{BG} \lim_{\sfK} H
    \simeq
    \lim_{\sfK} \lim_{BG} H
    \simeq
    \calB^{hG} \times_{\calD^{hG}} \calC^{hG}.
  \end{align}

  (7)
  Since the $G$-action on $\calC$ is trivial, there is a natural equivalence of $G$-categories
  \begin{align}
    \Fun^{\conj}((\calC,\triv),\calD)
    \simeq
    \Fun(\calC,-)\circ\calD
    : BG \to \widehat\Cat_{(\infty,1)}.
  \end{align}
  Since $\Fun(\calC,-)$ preserves limits, we have
  \begin{align}
    \Fun^{\conj}((\calC,\triv),\calD)^{hG}
    &\simeq
    \lim_{BG}(\Fun(\calC,-)\circ\calD) \\
    &\simeq
    \Fun(\calC,\lim_{BG}\calD) \\
    &=
    \Fun(\calC,\calD^{hG}).
  \end{align}
\end{proof}

\subsection{Anti-involutive \texorpdfstring{$(\infty,1)$}{(∞,1)}-categories}

We now specialize to $G=C_2$.
Taking homotopy fixed points of the opposite action defines anti-involutive $(\infty,1)$-categories, and the equivariant core supplies their Hermitian anima (\cref{target.def.ainv,target.def.hermitian-core}).

\begin{notation}
  Let $\rev:\prism\to\prism$ be the functor which is the identity on objects and sends $\alpha:[m]\to[n]$ to 
  \begin{align}
    \alpha^{\vee} := \bar\rho_n\circ\alpha^{\myop}\circ\bar\rho_m^{-1},
  \end{align}
  where $\bar\rho_n:[n]^{\myop}\to[n]$ is the order-preserving isomorphism $i\mapsto n-i$.
\end{notation}

\begin{notation}\label{rev.def.reversal-simplex}
  Let $\prism_{\rev}$ be the category whose objects are the finite nonempty ordinals and whose morphisms are the order-preserving or order-reversing maps.
\end{notation}

\begin{notation}
  We write $(X,\rho)$ for a $C_{2}$-action $(X,\rho): BC_{2} \to \calE$.
  Here $X$ is its value at $* \in BC_{2}$, while $\rho$ denotes the full coherent action datum.
  This includes the equivalences $\rho(\omega):X\xrightarrow{\simeq}X$, together with their unit, multiplication, and all higher coherence data.
\end{notation}

\begin{remark}
  The reversal action on simplicial anima is the conjugation construction:
  there is an equivalence of $C_2$-categories
  \begin{align}
    (\Fun(\prism^{\myop},\An),(\rev^{\myop})^{*})
    \simeq
    \Fun^{\conj}((\prism^{\myop},\rev^{\myop}),(\An,\triv)).
  \end{align}
\end{remark}

\begin{lemma}\label{rev.lem.reversal-segal}
  The following statements hold.
  \begin{enumerate}
    \item The $C_{2}$-category $(\Fun(\prism^{\myop},\An),(\rev^{\myop})^{*})$ restricts to a $C_{2}$-category
    \begin{align}
      (\Seg(\An),(\rev^{\myop})^{*}) 
      : BC_{2} \to \widehat{\Cat}_{(\infty,1)}.
    \end{align}
    \item The $C_2$-category $(\Seg(\An),(\rev^{\myop})^{*})$ restricts to a $C_2$-category
    \begin{align}
      (\CSeg(\An),(\rev^{\myop})^{*})
      : BC_2\to\widehat{\Cat}_{(\infty,1)}.
    \end{align}
    \item There is a $C_{2}$-category
    \begin{align}
      (\Cat_{(\infty,1)},(-)^{\myop}) : BC_{2} \to \widehat{\Cat}_{(\infty,1)}.
    \end{align}
    \item The $C_{2}$-category $(\Cat_{(\infty,1)},(-)^{\myop})$ restricts to a $C_{2}$-category
    \begin{align}
      (\An,\triv)
      \simeq 
      (\An,(-)^{\myop}) 
      : BC_{2} \to \widehat{\Cat}_{(\infty,1)}.
    \end{align}
  \end{enumerate}
\end{lemma}

\begin{proof}
  (1)
  For $n\geq2$, the functor $\rev$ carries the spine inclusion to the same spine with its consecutive edges in reverse order.
  Hence $X$ is Segal if and only if $X^{\rev}:=X\circ\rev^{\myop}$ is Segal.
  The full subcategory $\Seg(\An)$ therefore inherits the asserted action and the equivariant evaluation functor.

  (2)
  Let $X\in\Seg(\An)$.
  Write $e_{02},e_{13}:[1]\to[3]$ for the inclusions with images $\{0,2\}$ and $\{1,3\}$, respectively.
  For the unique map $p_n:[n]\to[0]$, write $c_n:=X(p_n):X_0\to X_n$.
  Then $X$ is complete if and only if the following square is a pullback in $\An$:
  \begin{align}
    Q_X:\quad
    \begin{tikzpicture}[auto,baseline=(current bounding box.center)]
      \node (tl) at (0, 1.5)
        {$X_0$};
      \node (tr) at (4, 1.5)
        {$X_3$};
      \node (bl) at (0, 0)
        {$X_0\times X_0$};
      \node (br) at (4, 0)
        {$X_1\times X_1$};
      \draw[->]
        (tl) to node
        {$\scriptstyle c_3$} (tr);
      \draw[->]
        (tl) to node[swap]
        {$\scriptstyle \Delta$} (bl);
      \draw[->]
        (tr) to node
        {$\scriptstyle (e_{02}^*,e_{13}^*)$} (br);
      \draw[->]
        (bl) to node
        {$\scriptstyle c_1\times c_1$} (br);
    \end{tikzpicture}
  \end{align}
  In $\prism$, we have $\rev(e_{02})=e_{13}$, $\rev(e_{13})=e_{02}$, and $\rev(p_n)=p_n$.
  Let $\tau_k:X_k\times X_k\to X_k\times X_k$ exchange the factors for $k=0,1$.
  Under the identification $X^{\rev}_{n}=X_n$, these identities give an equivalence $Q_{X^{\rev}}\simeq Q_X$ in $\Fun([1]\times[1],\An)$.
  Thus $Q_{X^{\rev}}$ is a pullback if and only if $Q_X$ is, so $X^{\rev}$ is complete if and only if $X$ is complete.
  Therefore the $C_{2}$-action on $\Seg(\An)$ restricts to that on $\CSeg(\An)$.

  (3)
  Under the equivalence $\CSeg(\An) \simeq \Cat_{(\infty,1)}$, the $C_{2}$-action $(\rev^{\myop})^{*}$ on $\CSeg(\An)$ corresponds to the opposite action $(-)^{\myop}$ on $\Cat_{(\infty,1)}$.

  (4)
  By \cite[Remark~2.5]{HLAS16}, the $C_{2}$-action on $\An$ is unique:
  there is an equivalence of $G$-categories 
  \begin{align}
    (\An,\triv)\simeq(\An,(-)^{\myop}) 
    : BC_{2} \to \widehat\Cat_{(\infty,1)}.
  \end{align}
\end{proof}

\begin{definition}\label{target.def.ainv}
  The $(\infty,1)$-category $\AInvCat_{(\infty,1)}$ of \emph{anti-involutive $(\infty,1)$-categories} is the homotopy fixed-point $(\infty,1)$-category of $(\Cat_{(\infty,1)},(-)^{\myop})$:
  \begin{align}
    \AInvCat_{(\infty,1)}
    :=
    (\Cat_{(\infty,1)}, (-)^{\myop})^{hC_{2}}.
  \end{align}
\end{definition}

\begin{remark}\label{target.not.ainv-cone-convention}
  An object of $\AInvCat_{(\infty,1)}$ is a cocartesian section of the opposite-action fibration over $BC_2$, by the description of limits of the $(\infty,1)$-category of cocartesian sections.
  It includes $d_{\calC}:\calC\simeq\calC^{\myop}$ and $d_{\calC}^{\myop}d_{\calC}\simeq\id_{\calC}$.
  The section also specifies the compatibility of this equivalence with $d_{\calC}$ and every higher $BC_2$-coherence.

  In this convention, a morphism $F:(\calC,d_{\calC}) \to (\calD,d_{\calD})$ in $\AInvCat_{(\infty,1)}$ contains an equivalence of functors $d_{\calD}F \simeq F^{\myop}d_{\calC} : \calC\to\calD^{\myop}$.
\end{remark}

\begin{lemma}\label{target.prop.equivariant-core}
  The following statements hold.
  \begin{enumerate}
    \item The adjunction $j: \An \rightleftarrows \Cat_{(\infty,1)} : \iota_{0}$ admits an equivariant refinement
    \begin{align}
      j:
      (\An,\triv)
      \simeq 
      (\An,(-)^{\myop})
      \rightleftarrows
      (\Cat_{(\infty,1)},(-)^{\myop}):
      \iota_0.
    \end{align}
    \item Taking homotopy fixed points gives an adjunction
    \begin{align}
      j^{hC_2}:
      \Fun(BC_2,\An)
      \rightleftarrows
      \AInvCat_{(\infty,1)}:
      \iota_0^{hC_2}.
    \end{align}
  \end{enumerate}
\end{lemma}

\begin{proof}
  (1)
  Compose the equivariant inclusion
  $(\An,(-)^{\myop})\hookrightarrow(\Cat_{(\infty,1)},(-)^{\myop})$
  with the equivalence $(\An,\triv)\simeq(\An,(-)^{\myop})$.
  Its underlying right adjoint is $\iota_0$.
  Hence \cref{target.lem.homotopy-fixed-point-formalism} (5) supplies the equivariant adjunction.

  (2)
  Apply \cref{target.lem.homotopy-fixed-point-formalism} (4) and identify $(\An,\triv)^{hC_2}$ with $\Fun(BC_2,\An)$ by \cref{target.lem.constant-groupoid-limit}.
  By \cref{target.lem.homotopy-fixed-point-formalism} (3), $j^{hC_2}$ is fully faithful, so its unit is an equivalence.
\end{proof}

\begin{definition}\label{target.def.hermitian-core}
  Define the \emph{Hermitian-core functor} $\Herm: \AInvCat_{(\infty,1)} \to \An$ by the composite 
  \begin{align}
    \Herm
    : \AInvCat_{(\infty,1)}
    \xrightarrow{\iota_0^{hC_2}}
    \Fun(BC_{2},\An) 
    \xrightarrow{(-)^{hC_2}}
    \An.
  \end{align}
\end{definition}

\begin{notation}\label{target.not.underlying-flag}
  For $\C \in \AInvCat_{(\infty,1)}$, write 
  \begin{align}
    \epsilon_{\calC}:
    \Herm(\calC)
    :=
    (\iota_0^{hC_{2}}\calC)^{hC_{2}}
    \to
    (\iota_0^{hC_{2}}\calC)_{*}
    \simeq 
    \iota_0\calC_{*}
  \end{align}
  for the value at $* \in BC_2$ of the counit $\Const_{BC_{2}}\Herm(\C) \to \iota_0^{hC_{2}}\calC$ of the adjunction $\const_{BC_{2}} \dashv (-)^{hC_{2}}$.
\end{notation}

\begin{lemma}
  For every $A \in \An$, there is an equivalence of anima 
  \begin{align}
    \Herm(j^{hC_{2}}\triv(A))
    \simeq 
    \Map_{\An}(BC_{2},A).
  \end{align}
\end{lemma}

\begin{proof}
  By \cref{target.prop.equivariant-core}, the unit of the adjunction $j^{hC_2}\dashv\iota_0^{hC_2}$ is an equivalence.
  By \cref{target.lem.constant-groupoid-limit}, we have
  \begin{align}
    \Herm(j^{hC_2}\triv(A))
    &=
    (\iota_0^{hC_2}j^{hC_2}\triv(A))^{hC_2}\\
    &\simeq
    \triv(A)^{hC_2} \\
    &=
    \lim_{BC_2}\Const_{BC_2}(A) \\
    &\simeq
    \Map_{\An}(BC_2,A).
  \end{align}
\end{proof}

\begin{lemma}\label{target.lem.equivariant-full-core}
  Let $F:\calC\to\calD$ be a morphism in $\AInvCat_{(\infty,1)}$ whose underlying functor $F_*:\calC_*\to\calD_*$ is fully faithful.
  Then the following statements hold.
  \begin{enumerate}
    \item The induced morphism $\Herm(F):\Herm(\calC)\to\Herm(\calD)$ is fully faithful.
    \item Its essential image consists of the Hermitian points whose underlying objects belong to the essential image of $F_*$.
  \end{enumerate}
\end{lemma}

\begin{proof}
  Since $F_{*}$ is fully faithful, so is $\iota_0F_* : \iota_{0}\calC_{*} \to \iota_{0}\calD_{*}$.
  Let $y\in\Herm(\calD)$ and let $\widetilde y:\triv(*)\to\iota_0^{hC_2}(\calD)$ be its adjunct.
  Form the pullback $P_y:=\triv(*)\times_{\iota_0^{hC_2}\calD}\iota_0^{hC_2}\calC$ in $\Fun(BC_2,\An)$.
  By \cref{target.lem.homotopy-fixed-point-formalism} (6), homotopy fixed points commute with pullbacks, so we have 
  \begin{align}
    \Herm(\C) \times_{\Herm(\D)} \{y\}
    \simeq 
    P_y^{hC_2}.
  \end{align}
  The underlying anima of $P_y$ is the fiber of $\iota_0F_*$ over $\epsilon_{\calD}(y)$.
  It is empty if that object is outside the essential image, and contractible otherwise.
  In the first case, $P_y^{hC_2}$ is empty.
  In the second case, $P_y\to\triv(*)$ is an objectwise equivalence, so $P_y^{hC_2}$ is contractible.

  Thus every fiber $\Herm(\C) \times_{\Herm(\D)} \{y\}$ is empty or contractible, so $\Herm(F)$ is fully faithful. 
  The description of its essential image follows from the same computation.
\end{proof}

\begin{proposition}\label{properties.prop.anti-involutive}
  The following statements hold.
  \begin{enumerate}
    \item The $(\infty,1)$-category $\AInvCat_{(\infty,1)}$ is presentable.
    \item The evaluation $\ev_*:\AInvCat_{(\infty,1)}\to\Cat_{(\infty,1)}$ creates small limits and colimits.
    \item The evaluation $\ev_*:\AInvCat_{(\infty,1)}\to\Cat_{(\infty,1)}$ admits a left adjoint $F_{\AI} :\Cat_{(\infty,1)} \to \AInvCat_{(\infty,1)}$.
    \item The functor $\Herm:\AInvCat_{(\infty,1)}\to\An$ is an accessible right adjoint, with left adjoint $j^{hC_2}\triv$.
  \end{enumerate}
\end{proposition}

\begin{proof}
  (1)
  Since the opposite action is an equivalence on $\Cat_{(\infty,1)}$, it preserves limits and colimits.
  The action therefore defines diagrams in both $\PrL$ and $\PrR$.
  Presentability of $\AInvCat_{(\infty,1)} = \lim_{BC_{2}}(\Cat_{(\infty,1)},(-)^{\myop})$ follows from \cite[Proposition~5.5.3.13 and Theorem~5.5.3.18]{HTT}.

  (2)
  By the same theorems, the evaluation $\ev_*$ preserves small limits and colimits.
  Since equivalences between cocartesian sections are detected pointwise, $\ev_*$ is conservative.
  Hence it also reflects limit cones and colimit cocones.
  Any underlying limit cone or colimit cocone is equivalent to the image of such a (co)cone in $\AInvCat_{(\infty,1)}$.

  (3)
  It follows from (1), (2), and the adjoint functor theorem \cite[Corollary~5.5.2.9 (2)]{HTT}.

  (4)
  There is an adjunction $j^{hC_2}\triv:\An\rightleftarrows\AInvCat_{(\infty,1)}:\Herm$.
  Both $(\infty,1)$-categories are presentable, so this right adjoint is accessible by \cite[Proposition~5.4.7.7]{HTT}.
\end{proof}

\subsection{Hermitian flags}

We introduce the $(\infty,1)$-category of Hermitian flags $i : A \to \Herm(\calC)$, where $A \in \An$ and $\calC \in \AInvCat_{(\infty,1)}$.

\begin{definition}\label{target.def.hermitian-flags}
  Define the $(\infty,1)$-category $\HermFlag_{(\infty,1)}$ of \emph{Hermitian flags} by the pullback in $\widehat{\Cat}_{(\infty,1)}$:
  \begin{align}
    \HermFlag_{(\infty,1)}
    :=
    \Fun([1],\An) \times_{\ev_1,\An,\Herm} \AInvCat_{(\infty,1)}.
  \end{align}
\end{definition}

\begin{remark}
  A Hermitian flag consists of a morphism of anima $i : A \to \Herm(\calC)$, where $A \in \An$ and $\C \in \AInvCat_{(\infty,1)}$.
  By \cref{target.not.underlying-flag}, there is a canonical \emph{underlying flag} $A \xrightarrow{i} \Herm(\calC) \xrightarrow{\epsilon_{\C}} \iota_{0}\C_{*}$.
\end{remark}

\begin{remark}\label{target.prop.mapping-spaces}
  Let $(i : A \to \Herm(\calC)), (j : B \to \Herm(\calD)) \in \HermFlag_{(\infty,1)}$.
  There is a natural equivalence
  \begin{align}
    \Map_{\HermFlag_{(\infty,1)}}(i,j)
    &\simeq
    \Map_{\Fun([1],\An)}(i,j)
    \times_{\Map_{\An}(\Herm(\calC),\Herm(\calD))}
    \Map_{\AInvCat_{(\infty,1)}}(\calC,\calD) \\
    &\simeq 
    \Map_{\An}(A,B) \times_{\Map_{\An}(A,\Herm(\calD))} \Map_{\AInvCat_{(\infty,1)}}(\calC,\calD).
  \end{align}
\end{remark}

\begin{proposition}\label{target.lem.ambient-presentability}
  The $(\infty,1)$-category $\HermFlag_{(\infty,1)}$ is presentable.
\end{proposition}

\begin{proof}
  By \cref{properties.prop.anti-involutive}, $\AInvCat_{(\infty,1)}$ is presentable and $\Herm$ is an accessible right adjoint.
  The functor $\ev_1:\Fun([1],\An)\to\An$ is right adjoint to $H\mapsto(\varnothing\to H)$.
  Thus the pullback of \cref{target.def.hermitian-flags} is a pullback in $\PrR$ and is presentable by \cite[Proposition~5.5.3.6 and Theorem~5.5.3.18]{HTT}.
\end{proof}

\begin{lemma}\label{target.lem.hermitian-flag-equivalences}
  The forgetful functors
  \begin{align}
    \HermFlag_{(\infty,1)} \to \An
    &: (i:A\to\Herm(\calC))\mapsto A, \\
    \HermFlag_{(\infty,1)} \to \Cat_{(\infty,1)}
    &: (i:A\to\Herm(\calC))\mapsto\calC_{*}
  \end{align}
  jointly detect equivalences.
\end{lemma}

\begin{proof}
  Let $(F,f):(i:A\to\Herm(\calC))\to(j:B\to\Herm(\calD))$ be a morphism in $\HermFlag_{(\infty,1)}$, with $F:\calC\to\calD$ and $f:A\to B$.
  By \cref{target.prop.mapping-spaces} and pointwise detection of equivalences in $\Fun([1],\An)$, $(F,f)$ is an equivalence if and only if $F$ and $f$ are equivalences.
  By \cref{properties.prop.anti-involutive} (2), $F$ is an equivalence if and only if its underlying functor $F_*$ is an equivalence.
\end{proof}

\begin{lemma}\label{target.lem.herm-forgetful-adjoints}
  The following statements hold.
  \begin{enumerate}
    \item The forgetful functor $\HermFlag_{(\infty,1)} \to \AInvCat_{(\infty,1)} : (i : A \to \Herm(\C)) \mapsto \C$ admits fully faithful left and right adjoints
    \begin{align}
      L_{\AI,\HermFlag} 
      &: \AInvCat_{(\infty,1)} \to \HermFlag_{(\infty,1)} 
      : \C \mapsto (! : \emptyset \to \Herm(\C)), \\
      R_{\AI,\HermFlag} 
      &:\AInvCat_{(\infty,1)} \to \HermFlag_{(\infty,1)} 
      : \C \mapsto (\id_{\Herm(\C)} : \Herm(\C) \to \Herm(\C)).
    \end{align} 
    \item The forgetful functor $\HermFlag_{(\infty,1)} \to \An : (i : A \to \Herm(\C)) \mapsto A$ admits a left adjoint and a fully faithful right adjoint 
    \begin{align}
      L_{\An,\HermFlag}
      &: \An \to \HermFlag_{(\infty,1)} 
      : A \mapsto (\eta_{A} : A \to \Herm(j^{hC_{2}}\triv(A))), \\
      R_{\An,\HermFlag}
      &: \An \to \HermFlag_{(\infty,1)} 
      : A \mapsto (! : A \to \Herm(\oneDag) \simeq *).
    \end{align} 
    \item The forgetful functor $\HermFlag_{(\infty,1)} \to \An : (i : A \to \Herm(\C)) \mapsto \Herm(\C)$ admits a left adjoint
    \begin{align}
      \overline{L}_{\An,\HermFlag}
      : \An \to \HermFlag_{(\infty,1)} 
      : A \mapsto (\emptyset \to \Herm(j^{hC_{2}}\triv(A))).
    \end{align} 
  \end{enumerate}
\end{lemma}

\begin{proof}
  (1) and (2) follow from \cref{target.prop.mapping-spaces}.

  (3) 
  This forgetful functor is the composite of the forgetful functor in (1) with $\Herm$.
  So it admits a left adjoint given by the composite of the left adjoint in (1) with $j^{hC_{2}}\triv$.
\end{proof}

\subsection{Univalent flagged dagger \texorpdfstring{$(\infty,1)$}{(∞,1)}-categories}

We impose the flagged and coflagged conditions on Hermitian flags and define univalent flagged dagger $(\infty,1)$-categories by requiring both.

\begin{definition}
  We define some full subcategories of $\HermFlag_{(\infty,1)}$.
  \begin{itemize}
    \item The full subcategory of \emph{flagged dagger $(\infty,1)$-categories}
    \begin{align}
      \FlDagCat_{(\infty,1)} 
      \subseteq 
      \HermFlag_{(\infty,1)}
    \end{align}
    spanned by the Hermitian flags $(i : A \to \Herm(\calC))$ whose underlying flag $\epsilon_{\calC}i : A \to \iota_0\calC_{*}$ is essentially surjective.
    \item The full subcategory of \emph{coflagged dagger $(\infty,1)$-categories}
    \begin{align}
      \coFlDagCat_{(\infty,1)} 
      \subseteq 
      \HermFlag_{(\infty,1)}
    \end{align}
    spanned by the Hermitian flags $(i : A \to \Herm(\calC))$ for which $i$ is fully faithful.
    \item The full subcategory of \emph{univalent flagged dagger $(\infty,1)$-categories} (or \emph{dagger $(\infty,1)$-categories})
    \begin{align}
      \DagCatInf_{(\infty,1)} 
      := \FlDagCat_{(\infty,1)} \cap \coFlDagCat_{(\infty,1)} 
      \subseteq \HermFlag_{(\infty,1)}.
    \end{align}
  \end{itemize}
\end{definition}

\begin{lemma}\label{target.lem.equivalences-of-flagged-and-coflagged}
  The following statements hold.
  \begin{enumerate}
    \item There is an equivalence of $(\infty,1)$-categories
    \begin{align}
      \FlDagCat_{(\infty,1)}
      \simeq
      \Fun([1],\AInvCat_{(\infty,1)})^{\estag}
      \times_{\ev_0,\AInvCat_{(\infty,1)},j^{hC_2}\triv}\An.
    \end{align}
    Here $\Fun([1],\AInvCat_{(\infty,1)})^{\estag} \subseteq \Fun([1],\AInvCat_{(\infty,1)})$ denotes the full subcategory spanned by morphisms whose underlying functors are essentially surjective.
    \item There is an equivalence of $(\infty,1)$-categories
    \begin{align}
      \coFlDagCat_{(\infty,1)}
      \simeq
      \Fun([1],\An)^{\fftag}
      \times_{\ev_1,\An,\Herm}\AInvCat_{(\infty,1)}.
    \end{align}
  \end{enumerate}
\end{lemma}

\begin{proof}
  (1)
  By \cref{properties.prop.anti-involutive} (4), the adjunction $j^{hC_2}\triv\dashv\Herm$ identifies a Hermitian flag $X=(i:A\to\Herm(\calC))$ with its adjunct $i^{\sharp}:j^{hC_2}\triv(A)\to\calC$.
  Then we have
  \begin{align}
    \HermFlag_{(\infty,1)}
    &=
    \Fun([1],\An)\times_{\ev_1,\An,\Herm}\AInvCat_{(\infty,1)} \\
    &\simeq
    \Fun([1],\AInvCat_{(\infty,1)})
    \times_{\ev_0,\AInvCat_{(\infty,1)},j^{hC_2}\triv}\An.
  \end{align}
  The underlying functor $(i^{\sharp})_*:jA\to\calC_*$ is adjoint to the underlying flag $\epsilon_{\calC}i$ under $j\dashv\iota_0$.
  Thus $(i^{\sharp})_*$ is essentially surjective if and only if $X$ is flagged.
  Therefore the above equivalence restricts to an equivalence
  \begin{align}
    \FlDagCat_{(\infty,1)}
    \simeq
    \Fun([1],\AInvCat_{(\infty,1)})^{\estag}
    \times_{\ev_0,\AInvCat_{(\infty,1)},j^{hC_2}\triv}\An.
  \end{align}

  (2)
  By definition, a Hermitian flag $i:A\to\Herm(\calC)$ is coflagged if and only if $i$ is fully faithful.
  Hence we have 
  \begin{align}
    \coFlDagCat_{(\infty,1)}
    \simeq
    \Fun([1],\An)^{\fftag}
    \times_{\ev_1,\An,\Herm}\AInvCat_{(\infty,1)}.
  \end{align}
\end{proof}

\begin{lemma}
  The following statements hold.
  \begin{enumerate}
    \item The forgetful functor $\FlDagCat_{(\infty,1)} \to \AInvCat_{(\infty,1)} : (i : A \to \Herm(\C)) \mapsto \C$ admits a right adjoint 
    \begin{align}
      R_{\flag}R_{\AI,\HermFlag}
      :\AInvCat_{(\infty,1)}\to\FlDagCat_{(\infty,1)}.
    \end{align}
    \item The forgetful functor $\coFlDagCat_{(\infty,1)} \to \AInvCat_{(\infty,1)} : (i : A \to \Herm(\C)) \mapsto \C$ admits fully faithful left and right adjoints
    \begin{align}
      L_{\AI,\HermFlag} 
      : \AInvCat_{(\infty,1)} \to \coFlDagCat_{(\infty,1)} 
      \subseteq \HermFlag_{(\infty,1)}, \\
      R_{\AI,\HermFlag} 
      : \AInvCat_{(\infty,1)} \to \coFlDagCat_{(\infty,1)} 
      \subseteq \HermFlag_{(\infty,1)}.
    \end{align} 
    \item The forgetful functor $\FlDagCat_{(\infty,1)} \to \An : (i : A \to \Herm(\C)) \mapsto A$ admits left and right adjoints
    \begin{align}
      L_{\An,\HermFlag}
      &: \An\to\FlDagCat_{(\infty,1)} 
      \subseteq \HermFlag_{(\infty,1)},\\
      R_{\flag}R_{\An,\HermFlag}
      &: \An\to\FlDagCat_{(\infty,1)}.
    \end{align}
    \item The forgetful functor $\coFlDagCat_{(\infty,1)} \to \An : (i : A \to \Herm(\C)) \mapsto A$ admits a left adjoint
    \begin{align}
      L_{\co\flag}L_{\An,\HermFlag}
      :\An\to\coFlDagCat_{(\infty,1)}.
    \end{align}
    \item The forgetful functor $\coFlDagCat_{(\infty,1)} \to \An : (i : A \to \Herm(\C)) \mapsto \Herm(\C)$ admits a left adjoint
    \begin{align}
      \overline{L}_{\An,\HermFlag} 
      : \An \to \coFlDagCat_{(\infty,1)}
      \subseteq \HermFlag_{(\infty,1)}.
    \end{align}
  \end{enumerate}
\end{lemma}

\begin{proof}
  These statements follow from \cref{target.lem.herm-forgetful-adjoints,flag.prop.ambient-coflagged-reflection,flag.prop.ambient-flagged-coreflection}. 
\end{proof}

\begin{remark}
  Every $\calC \in \AInvCat_{(\infty,1)}$ has a canonical coflag $(\calC,\Herm(\calC),\id_{\Herm(\calC)})$.
  Its underlying flag is the forgetful morphism $\epsilon_{\calC}:\Herm(\calC) \to \iota_{0}\calC_{*}$.
  Hence this canonical coflag is flagged if and only if every object of $\calC$ \emph{admits a Hermitian structure}.
\end{remark}

\section{The flagged nerve construction}\label{direct.section}

In this section, we construct the realization--nerve adjunction on all Hermitian flags (\cref{flagged-nerve.lem.adjunction}).
Computing its unit and counit identifies flagged dagger $(\infty,1)$-categories with dagger Segal anima (\cref{direct.thm.flagged-equivalence}).
We then use the walking unitary to identify the coflagged condition with unitary completeness.

\subsection{Dagger Segal anima}

The walking flagged dagger simplices combine an ordinary chain with canonical Hermitian structures on its vertices (\cref{direct.construction.nerve}).
Their restricted Yoneda functor is the flagged nerve.

\begin{definition}
  A \emph{reversal simplicial anima} is a functor $W : \prism_{\rev}^{\myop} \to \An$.
  We let $\Fun(\prism_{\rev}^{\myop},\An)$ denote the $(\infty,1)$-category of reversal simplicial anima.
\end{definition}

\begin{definition}
  Let $W\in\Fun(\prism_{\rev}^{\myop},\An)$.
  We will say that it is a \emph{dagger Segal anima} if its restriction to $\prism^{\myop}$ satisfies the ordinary Segal condition.
\end{definition}

\begin{notation}
  We let 
  \begin{align}
    \SegRev(\An) \subseteq \Fun(\prism_{\rev}^{\myop},\An)
  \end{align}
  denote the full subcategory spanned by dagger Segal anima.
\end{notation}

\begin{construction}\label{direct.construction.nerve}
  For every $[n] \in \prism_{\rev}$, define the anti-involutive $(\infty,1)$-category $[n]_{\AI}$ by the following pushout square in $\AInvCat_{(\infty,1)}$:
  \begin{align}
    \begin{tikzpicture}[auto,baseline=(current bounding box.center)]
      \node (tl) at (0,1.5) {$F_{\AI}(j\iota_{0}[n])$};
      \node (tr) at (4,1.5) {$j^{hC_{2}}\triv(\iota_{0}[n])$};
      \node (bl) at (0,0) {$F_{\AI}([n])$};
      \node (br) at (4,0) {$[n]_{\AI}$};
      \draw[->] (tl) to node {} (tr);
      \draw[->] (tl) to node[swap] {} (bl);
      \draw[->] (tr) to node {$\scriptstyle i_{n}^{\sharp}$} (br);
      \draw[->] (bl) to (br);
    \end{tikzpicture}
  \end{align}
  Here the left vertical map is induced by the inclusion $j\iota_{0}[n] \hookrightarrow [n]$, and the upper horizontal map is induced by the counit 
  \begin{align}
    F_{\AI}(j\iota_{0}[n])
    \simeq 
    F_{\AI}(j^{hC_{2}}\triv(\iota_{0}[n])_{*})
    \to 
    j^{hC_{2}}\triv(\iota_{0}[n]).
  \end{align}

  For every $[n] \in \prism_{\rev}$, define a Hermitian flag $[n]_{\flag}$ by
  \begin{align}
    [n]_{\flag}
    := 
    ([n]_{\AI}, \iota_{0}[n], i_{n} : \iota_{0}[n] \to \Herm([n]_{\AI})),
  \end{align}
  where $i_{n}$ is the adjunct of $i_{n}^{\sharp}$ under the adjunction $j^{hC_{2}}\triv \dashv \Herm$.
  This construction defines a functor
  \begin{align}
    [-]_{\flag}
    : \prism_{\rev}
    \to
    \HermFlag_{(\infty,1)}.
  \end{align}
\end{construction}

\begin{lemma}\label{direct.lem.walking-model}
  For every $[n] \in \prism_{\rev}^{\myop}$, there is an equivalence of $(\infty,1)$-categories
  \begin{align}
    ([n]_{\AI})_{*}
    \simeq
    [n] \amalg_{j\iota_0[n]} [n]^{\myop}.
  \end{align}
\end{lemma}

\begin{proof}
  For every $\calC\in\Cat_{(\infty,1)}$, the free anti-involutive $(\infty,1)$-category satisfies $F_{\AI}(\calC)_*\simeq\calC\amalg\calC^{\myop}$.
  Since $(-)_{*} : \AInvCat_{(\infty,1)} \to \Cat_{(\infty,1)}$ preserves colimits, we have 
  \begin{align}
    ([n]_{\AI})_{*}
    &\simeq 
    ([n] \amalg [n]^{\myop}) \amalg_{j\iota_{0}[n] \amalg j\iota_{0}[n]} j\iota_{0}[n] \\ 
    &\simeq 
    [n] \amalg_{j\iota_{0}[n]} [n]^{\myop}.
  \end{align}
\end{proof}

\begin{definition}\label{flagged-nerve.construction.nerve}
  For $X = (i : A \to \Herm(\calC)) \in \HermFlag_{(\infty,1)}$, define the \emph{flagged nerve} $\rmN_{\flag}(X) : \prism_{\rev}^{\myop} \to \An$ of $X$ by the following formula for every $[n] \in \prism_{\rev}^{\myop}$:
  \begin{align}
    \rmN_{\flag}(X)_{n}
    &:= 
    \Map_{\HermFlag_{(\infty,1)}}([n]_{\flag},X) \\
    &\simeq
    \Map_{\An}(\iota_{0}[n],A) \times_{\Map_{\An}(\iota_{0}[n],\Herm(\C))} \Map_{\AInvCat_{(\infty,1)}}([n]_{\AI},\C).
  \end{align}
  This construction defines the \emph{flagged-nerve functor} 
  \begin{align}
    \rmN_{\flag} 
    : \HermFlag_{(\infty,1)} 
    \to 
    \Fun(\prism_{\rev}^{\myop},\An).
  \end{align}
\end{definition}

\begin{lemma}\label{flagged-nerve.lem.value-of-flagged-nerve}
\hypersetup{linkcolor=blue,citecolor=blue,urlcolor=blue}
  For every $X = (i : A \to \Herm(\calC)) \in \HermFlag_{(\infty,1)}$, there is a natural equivalence of reversal simplicial anima
  \begin{align}
    \rmN_{\flag}(X)_{\bullet}
    \simeq
    A^{\bullet+1} \times_{(\epsilon_{\calC}i)^{\bullet+1},(\iota_0\calC_{*})^{\bullet+1},(\ev_{0},\cdots,\ev_{n})} \iota_0 \Fun([\bullet],\calC_{*}).
  \end{align}
\end{lemma}

\begin{proof}
  For every $[n] \in \prism_{\rev}$, by \cref{target.prop.mapping-spaces}, we have
  \begin{align}
    \rmN_{\flag}(X)_{n}
    &\simeq
    \Map_{\An}(\iota_{0}[n],A) \times_{\Map_{\An}(\iota_{0}[n],\Herm(\C))} \Map_{\AInvCat_{(\infty,1)}}([n]_{\AI},\C) \\
    &\simeq
    \Map_{\An}(\iota_{0}[n],A) \times_{\Map_{\AInvCat_{(\infty,1)}}(F_{\AI}(j\iota_{0}[n]),\C)} \Map_{\AInvCat_{(\infty,1)}}(F_{\AI}([n]),\C) \\
    &\simeq
    \Map_{\An}(\iota_{0}[n],A) \times_{\Map_{\Cat_{(\infty,1)}}(j\iota_{0}[n],\C_{*})} \Map_{\Cat_{(\infty,1)}}([n],\C_{*}) \\
    &\simeq
    A^{n+1} \times_{(\iota_{0}\C_{*})^{n+1}} \iota_{0}\Fun([n],\C_{*}).
  \end{align}
\end{proof}

\begin{lemma}\label{flagged-nerve.lem.segal}
  The flagged nerve $\rmN_{\flag} : \HermFlag_{(\infty,1)} \to \Fun(\prism_{\rev}^{\myop},\An)$ factors through $\SegRev(\An)$:
  \begin{align}
    \rmN_{\flag} : 
    \HermFlag_{(\infty,1)}
    \to \SegRev(\An)
    \subseteq \Fun(\prism_{\rev}^{\myop},\An).
  \end{align}
\end{lemma}

\begin{proof}
  Let $X = (i : A \to \Herm(\calC)) \in \HermFlag_{(\infty,1)}$.
  The ordinary nerve $[n] \mapsto \rmN(\calC)_{n}$ defines a Segal anima.
  Pulling back this Segal equivalence along the morphism in \cref{flagged-nerve.lem.value-of-flagged-nerve} gives an equivalence
  \begin{align}
    (\rmN_{\flag}X)_n
    \simeq
    (\rmN_{\flag}X)_1 \times_{(\rmN_{\flag}X)_0} \cdots \times_{(\rmN_{\flag}X)_0} (\rmN_{\flag}X)_1.
  \end{align}
\end{proof}

\begin{lemma}\label{flagged-nerve.lem.adjunction}
  Let $\yo_{\rev} : \prism_{\rev} \to \Fun(\prism_{\rev}^{\myop},\An)$ be the Yoneda embedding.
  There is an adjunction
  \begin{align}
    \ac_{\flag} := 
    \Lan_{\yo_{\rev}}([-]_{\flag})
    : \Fun(\prism_{\rev}^{\myop},\An)
    \rightleftarrows
    \HermFlag_{(\infty,1)}
    : \rmN_{\flag}.
  \end{align}
\end{lemma}

\begin{proof}
  The $(\infty,1)$-category $\HermFlag_{(\infty,1)}$ is presentable by \cref{target.lem.ambient-presentability}.
  Hence the universal property of presheaves \cite[Theorem~5.1.5.6]{HTT} gives the left Kan extension and its right adjoint,
  which is the restricted Yoneda functor of \cref{flagged-nerve.construction.nerve}.
\end{proof}

\subsection{Calculating unit and counit}

We compute the underlying realization using colimits of representables (\cref{direct.lem.underlying-realization}).
We then identify the unit and counit using ordinary Segal completion (\cref{direct.thm.unit,direct.thm.counit}).

\begin{notation}\label{direct.not.categorical-realization}
  We fix some notational conventions.
  \begin{itemize}
    \item Let $\iota_{\rev}:\prism\hookrightarrow\prism_{\rev}$ denote the inclusion of the order-preserving maps.
    \item Let
    \begin{align}
      \ac:
      \Fun(\prism^{\myop},\An)
      \rightleftarrows
      \Cat_{(\infty,1)}
      :\rmN
    \end{align}
    denote the realization--nerve adjunction, where $\rmN$ is given by $\rmN(\calC)_n:=\iota_0\Fun([n],\calC)$.
  \end{itemize}
\end{notation}

\begin{lemma}\label{direct.lem.walking-flag-kan-extension}
  For every $[n]\in\prism_{\rev}$, there is an equivalence in $\Fun([1],\Cat_{(\infty,1)})$, natural in $[n]\in\prism_{\rev}$,
  \begin{align}
    \colim_{([m],\beta)\in\iota_{\rev}/[n]}
    (j\iota_0[m]\to[m])
    \simeq
    (j\iota_0[n]\to([n]_{\AI})_*).
  \end{align}
\end{lemma}

\begin{proof}
  For each $([m],\beta)\in\iota_{\rev}/[n]$, the forward-chain inclusion of \cref{direct.construction.nerve} gives the composite
  \begin{align}
    [m]
    \to
    ([m]_{\AI})_*
    \xrightarrow{([-]_{\flag}(\beta))_*}
    ([n]_{\AI})_*.
  \end{align}
  Together with the induced morphism $j\iota_0[m]\to j\iota_0[n]$, these composites define a comparison morphism
  \begin{align}
    \colim_{([m],\beta)\in\iota_{\rev}/[n]}
    (j\iota_0[m]\to[m])
    \to
    (j\iota_0[n]\to([n]_{\AI})_*).
  \end{align}
  We show that this comparison is an equivalence at both endpoints.

  At the source endpoint, we have 
  \begin{align}
    \colim_{([m],\beta)\in\iota_{\rev}/[n]}j\iota_0[m]
    &\simeq
    j(\colim_{([m],\beta)\in\iota_{\rev}/[n]}\rmN([m])_0) \\
    &\simeq
    j((\iota_{\rev}^{\myop,*}\yo_{\rev}([n]))_0) \\
    &\simeq
    j\iota_0[n].
  \end{align}
  Under this identification, the source comparison is the identity of $j\iota_0[n]$.

  At the target endpoint, order-preserving and order-reversing maps give a pushout of presheaves
  \begin{align}
    \iota_{\rev}^{\myop,*}\yo_{\rev}([n])
    \simeq
    \rmN([n])
    \amalg_{\Const_{\iota_0[n]}}
    \rmN([n]^{\myop}).
  \end{align}
  The intersection consists of the constant maps.
  Its inclusions are injections of discrete anima, so the degreewise pushout is also a pushout in anima.
  Moreover, $\Const_{\iota_0[n]}\simeq\rmN(j\iota_0[n])$.
  By \cref{direct.lem.walking-model}, we have 
  \begin{align}
    \colim_{([m],\beta)\in\iota_{\rev}/[n]}[m]
    &\simeq
    \ac(\iota_{\rev}^{\myop,*}\yo_{\rev}([n]))\\
    &\simeq
    [n]\amalg_{j\iota_0[n]}[n]^{\myop}\\
    &\simeq
    ([n]_{\AI})_*.
  \end{align}

  For order-preserving $\beta$, the target cocone is the composite of $[m]\xrightarrow{\beta}[n]$ with the forward-chain inclusion.
  For order-reversing $\beta$, it is the composite of $[m]\xrightarrow{\beta}[n]^{\myop}$ with the reverse-chain inclusion.
  These descriptions agree on constant morphisms through the common vertices.
  Thus the target comparison is induced by the identity on both chains and their common vertices, so it is the displayed equivalence.
  This equivalence is natural in $[n]$.
\end{proof}

\begin{lemma}\label{direct.lem.underlying-realization}
  Let $W\in\Fun(\prism_{\rev}^{\myop},\An)$ and write $\ac_{\flag}W=(i_W:A_W\to\Herm(\calC_W))$.
  The following statements hold.
  \begin{enumerate}
    \item There is an equivalence in $\Fun([1],\Cat_{(\infty,1)})$, natural in $W$,
    \begin{align}
      \colim_{\rmN([m])\to\iota_{\rev}^{\myop,*}W}
      (j\iota_0[m]\to[m])
      \simeq
      (jA_W\to(\calC_W)_*).
    \end{align}
    \item The endpoint identifications in (1) give a commutative square, natural in $W$, whose horizontal arrows are equivalences:
    \begin{align}
      \begin{tikzpicture}[auto,baseline=(current bounding box.center)]
        \node (tl) at (0,1.5) {$jW_0$};
        \node (tr) at (4,1.5) {$jA_W$};
        \node (bl) at (0,0)
          {$\ac(\iota_{\rev}^{\myop,*}W)$};
        \node (br) at (4,0) {$(\calC_W)_*$};
        \draw[->] (tl) to
          node {$j(\eta^{\flag}_W)_0$}
          node[swap] {$\simeq$} (tr);
        \draw[->] (tl) to
          node[swap]
          {$((\eta_{\iota_{\rev}^{\myop,*}W})_0)^{\sharp}$}
          (bl);
        \draw[->] (tr) to
          node {$(\epsilon_{\calC_W}i_W)^{\sharp}$}
          (br);
        \draw[->] (bl) to
          node {$\theta_W$}
          node[swap] {$\simeq$} (br);
      \end{tikzpicture}
    \end{align}
    Thus the underlying flag of $\ac_{\flag}W$ is naturally equivalent to the flag given by the degree-zero part of the ordinary realization unit.
  \end{enumerate}
\end{lemma}

\begin{proof}
  We prove (1) and (2) together.
  For $X=(i:A\to\Herm(\calC))$, the second projection of \cref{flagged-nerve.lem.value-of-flagged-nerve} defines a natural morphism
  \begin{align}
    \pr_{X}:
    \iota_{\rev}^{\myop,*}\rmN_{\flag}X
    \to
    \rmN(\calC_*).
  \end{align}
  Define $\theta_W:\ac(\iota_{\rev}^{\myop,*}W)\to(\calC_W)_*$ to be the adjunct under the adjunction $\ac\dashv\rmN$ of
  \begin{align}
    \iota_{\rev}^{\myop,*}W
    \xrightarrow{\iota_{\rev}^{\myop,*}\eta^{\flag}_W}
    \iota_{\rev}^{\myop,*}\rmN_{\flag}\ac_{\flag}W
    \xrightarrow{\pr_{\ac_{\flag}W}}
    \rmN(\calC_W)_*.
  \end{align}
  It is natural in $W$, and its defining mate identity is
  \begin{align}
    \rmN(\theta_W)\eta_{\iota_{\rev}^{\myop,*}W}
    \simeq
    \pr_{\ac_{\flag}W}
    \iota_{\rev}^{\myop,*}\eta^{\flag}_W.
  \end{align}
  In degree zero, this gives
  \begin{align}
    \iota_0(\theta_W)
    (\eta_{\iota_{\rev}^{\myop,*}W})_0
    \simeq
    \epsilon_{\calC_W}i_W(\eta^{\flag}_W)_0.
  \end{align}
  Taking adjuncts under the adjunction $j\dashv\iota_0$ constructs the square in (2).

  Evaluation at $[0]$ and the left adjoints $j$ and $\ac$ preserve colimits.
  Then we have 
  \begin{align}
    \colim_{\rmN([m])\to\iota_{\rev}^{\myop,*}W}
    j\iota_0[m]
    \simeq jW_0
    \quad \text{and} \quad 
    \colim_{\rmN([m])\to\iota_{\rev}^{\myop,*}W}
    [m]
    \simeq
    \ac(\iota_{\rev}^{\myop,*}W).
  \end{align}

  For a morphism $\sigma:\rmN([m])\to\iota_{\rev}^{\myop,*}W$, the induced morphism $[m]\to\ac(\iota_{\rev}^{\myop,*}W)$ is adjoint to the composite $\eta_{\iota_{\rev}^{\myop,*}W}\sigma$ under the adjunction $\ac\dashv\rmN$.
  Its restriction to $j\iota_0[m]$ is therefore adjoint to
  \begin{align}
    \iota_0[m]
    \xrightarrow{\sigma_0}
    W_0
    \xrightarrow{(\eta_{\iota_{\rev}^{\myop,*}W})_0}
    \iota_0\ac(\iota_{\rev}^{\myop,*}W).
  \end{align}
  Consequently, the preceding endpoint identifications give an equivalence of arrows
  \begin{align}
    \colim_{\rmN([m])\to\iota_{\rev}^{\myop,*}W}
    (j\iota_0[m]\to[m])
    \simeq
    (jW_0
    \xrightarrow{((\eta_{\iota_{\rev}^{\myop,*}W})_0)^{\sharp}}
    \ac(\iota_{\rev}^{\myop,*}W)).
  \end{align}
  Composing with the square in (2) defines a natural comparison
  \begin{align}
    \colim_{\rmN([m])\to\iota_{\rev}^{\myop,*}W}
    (j\iota_0[m]\to[m])
    \to
    (jA_W
    \xrightarrow{(\epsilon_{\calC_W}i_W)^{\sharp}}
    (\calC_W)_*).
  \end{align}

  Both sides of this comparison preserve colimits in $W$.
  For the left side, the endpoint functors are $W\mapsto jW_0$ and $W\mapsto\ac(\iota_{\rev}^{\myop,*}W)$.
  They preserve colimits.
  For the right side, $\ac_{\flag}$ is a left adjoint, and the projections from $\HermFlag_{(\infty,1)}$ to $\An$ and $\AInvCat_{(\infty,1)}$ preserve colimits by \cref{target.lem.herm-forgetful-adjoints} (1), (2).
  Moreover, $j$ preserves colimits, and $(-)_*:\AInvCat_{(\infty,1)}\to\Cat_{(\infty,1)}$ preserves colimits by \cref{properties.prop.anti-involutive} (2).

  Then the unit at $\yo_{\rev}([n])$ corresponds to $\id_{[n]_{\flag}}$, and its degree-zero part is the identity of $\iota_0[n]$.
  Its composite with $\pr_{[n]_{\flag}}$ sends $\beta:\iota_{\rev}([m])\to[n]$ to $[m] \to ([m]_{\AI})_* \xrightarrow{([-]_{\flag}(\beta))_*} ([n]_{\AI})_*$.
  Hence the comparison on $\yo_{\rev}([n])$ is precisely the equivalence of \cref{direct.lem.walking-flag-kan-extension}.

  It follows that the comparison is an equivalence for every $W$, proving (1).
  Its endpoint maps $j(\eta^{\flag}_W)_0$ and $\theta_W$ are therefore equivalences, proving (2).
  Since $j$ is fully faithful, $(\eta^{\flag}_W)_0$ is itself an equivalence of anima.
\end{proof}

\begin{lemma}\label{direct.lem.ordinary-completion}
  The following statements hold.
  \begin{enumerate}
    \item For every Segal anima $V:\prism^{\myop}\to\An$, the degree-zero unit $(\eta_V)_0:V_0\to\iota_0\ac V$ is essentially surjective.
    \item For every Segal anima $V:\prism^{\myop}\to\An$, the vertex maps and the unit $\eta_V$ induce an equivalence of anima, natural in $V$ and $[n]\in\prism$,
    \begin{align}
      V_n\xrightarrow{\simeq}
      V_0^{n+1}\times_{(\iota_0\ac V)^{n+1}}
      \iota_0\Fun([n],\ac V).
    \end{align}
    \item More generally, let $u:A\to\iota_0\calC$ and define a simplicial anima $V : \prism^{\myop} \to \An$ by
     \begin{align}
      V_n:=A^{n+1}\times_{(\iota_0\calC)^{n+1}}\iota_0\Fun([n],\calC).
    \end{align}
    Then the functor $\ac V\to\calC$ adjoint to the second projection $V\to\rmN(\calC)$ is fully faithful, and its essential image consists of the objects equivalent to some $u(a)$.
  \end{enumerate}
\end{lemma}

\begin{proof}
  (1)
  Since $\rmN$ identifies $\Cat_{(\infty,1)}$ with the complete Segal anima, $\eta_V:V\to\rmN (\ac V)$ is the ordinary completion unit.
  It is a Dwyer--Kan equivalence by \cite[Theorem~7.7]{Rez01}.
  Thus $(\eta_V)_0$ is essentially surjective.

  (2)
  Full faithfulness of the same Dwyer--Kan equivalence implies that the following square is a pullback square in $\An$:
  \begin{align}
    \begin{tikzpicture}[auto,baseline=(current bounding box.center)]
      \node (tl) at (0,1.5) {$V_1$};
      \node (tr) at (4.5,1.5) {$\iota_0\Fun([1],\ac V)$};
      \node (bl) at (0,0) {$V_0\times V_0$};
      \node (br) at (4.5,0) {$(\iota_0\ac V)\times(\iota_0\ac V)$};
      \draw[->] (tl) to node {$\scriptstyle(\eta_V)_1$} (tr);
      \draw[->] (tl) to node[swap] {$\scriptstyle(d_1,d_0)$} (bl);
      \draw[->] (tr) to node {$\scriptstyle(\ev_0,\ev_1)$} (br);
      \draw[->] (bl) to node {$\scriptstyle(\eta_V)_0\times(\eta_V)_0$} (br);
    \end{tikzpicture}
  \end{align}
  Applying the Segal equivalences for $V$ and $\rmN(\ac V)$ therefore gives, for every $n\geq 2$,
  \begin{align}
    V_n
    \xrightarrow{\simeq}
    V_0^{n+1}\times_{(\iota_0\ac V)^{n+1}}
    \iota_0\Fun([n],\ac V).
  \end{align}
  The case $n=0$ is immediate, and the case $n=1$ is the pullback above.
  All these maps are the comparisons induced by the vertices and $\eta_V$, so they are natural in $V$ and $[n]$.

  (3)
  The second projection $V\to\rmN(\calC)$ factors through the nerve of the full subcategory of $\calC$ on the objects equivalent to some $u(a)$.
  Base change of the Segal equivalences shows that $V$ is Segal, and the factorization is essentially surjective on objects.
  For $a,b\in A\simeq V_0$, it induces the equivalence $\Map_{V}(a,b)\simeq\Map_{\calC}(u(a),u(b))$.
  Hence this factorization is a Dwyer--Kan equivalence, so its completion is an equivalence onto that full subcategory by \cite[Theorem~7.7]{Rez01}.
  Composing with the full inclusion into $\calC$ gives the functor adjoint to the second projection.
  This proves its full faithfulness and the asserted description of its essential image.
\end{proof}

\begin{proposition}\label{direct.thm.unit}
  Let $W\in\SegRev(\An)$ and write $\ac_{\flag}W=(i_W:A_W\to\Herm(\calC_W))$.
  The following statements hold.
  \begin{enumerate}[beginpenalty=10000,midpenalty=10000]
    \item The unit $\eta^{\flag}_W:W\to\rmN_{\flag}\ac_{\flag}W$ of the adjunction in \cref{flagged-nerve.lem.adjunction} is an equivalence.
    \item $\ac_{\flag}W$ is flagged.
  \end{enumerate}
\end{proposition}

\begin{proof}
  (1)
  The equivalences $(\eta^{\flag}_W)_0:W_0\xrightarrow{\simeq}A_W$ and $\theta_W:\ac(\iota_{\rev}^{\myop,*}W)\xrightarrow{\simeq}(\calC_W)_*$ of \cref{direct.lem.underlying-realization} (2) induce the second arrow in
  \begin{align}
    W_n
    &\to
    W_0^{n+1}\times_{(\iota_0\ac(\iota_{\rev}^{\myop,*}W))^{n+1}}
    \iota_0\Fun([n],\ac(\iota_{\rev}^{\myop,*}W)) \\*
    &\xrightarrow{\simeq}
    A_W^{n+1}\times_{(\iota_0(\calC_W)_*)^{n+1}}
    \iota_0\Fun([n],(\calC_W)_*) \\
    &\simeq
    (\rmN_{\flag}\ac_{\flag}W)_n.
  \end{align}
  The first arrow is the ordinary completion comparison, and the last equivalence is \cref{flagged-nerve.lem.value-of-flagged-nerve}.
  The defining mate identity for $\theta_W$ in the proof of \cref{direct.lem.underlying-realization} identifies this composite with $(\eta^{\flag}_W)_n$.
  Since $\iota_{\rev}^{\myop,*}W$ is Segal, the first arrow is an equivalence by \cref{direct.lem.ordinary-completion} (2).

  Thus $\iota_{\rev}^{\myop,*}\eta^{\flag}_W$ is an equivalence.
  Since $\iota_{\rev}^{\myop}$ is the identity on objects, its restriction functor is conservative.
  Hence $\eta^{\flag}_W$ is an equivalence in $\Fun(\prism_{\rev}^{\myop},\An)$.

  (2)
  By \cref{direct.lem.ordinary-completion} (1), the morphism $(\eta_{\iota_{\rev}^{\myop,*}W})_0:W_0\to\iota_0\ac(\iota_{\rev}^{\myop,*}W)$ is essentially surjective.
  By \cref{direct.lem.underlying-realization} (2), we have
  \begin{align}
    \epsilon_{\calC_W}i_W (\eta^{\flag}_W)_0
    \simeq
    \iota_0(\theta_W) (\eta_{\iota_{\rev}^{\myop,*}W})_0
    : W_0
    \to 
    \iota_{0}(\calC_{W})_{*}.
  \end{align}
  Both $(\eta^{\flag}_W)_0$ and $\iota_0(\theta_W)$ are equivalences.
  Therefore $\epsilon_{\calC_W}i_W$ is essentially surjective, so $\ac_{\flag}W$ is flagged.
\end{proof}

\begin{proposition}\label{direct.thm.counit}
  Let $X=(i:A\to\Herm(\calC))\in\HermFlag_{(\infty,1)}$.
  Write $\varepsilon^{\flag}_{X}:\ac_{\flag}\rmN_{\flag}X\to X$ for the counit of $\ac_{\flag}\dashv\rmN_{\flag}$ at $X$.
  The following statements hold.
  \begin{enumerate}
    \item Under the degree-zero identifications of \cref{flagged-nerve.lem.value-of-flagged-nerve} and \cref{direct.lem.underlying-realization} (2), $\varepsilon^{\flag}_{X}$ induces the identity of $A$.
    \item On underlying anti-involutive $(\infty,1)$-categories, $\varepsilon^{\flag}_{X}$ identifies its source with the full subcategory
    \begin{align}
      (\calC_{*})_A
      :=\{x\in\calC_{*}\mid x\simeq\epsilon_{\calC}i(a)\text{ for some }a\in A\}
      \subseteq\calC_{*}.
    \end{align}
    \item $\varepsilon^{\flag}_{X}$ is an equivalence if and only if $X$ is flagged.
  \end{enumerate}
\end{proposition}

\begin{proof}
  Write $\ac_{\flag}\rmN_{\flag}X=(i_{\rmN_{\flag}X}:A_{\rmN_{\flag}X}\to\Herm(\calC_{\rmN_{\flag}X}))$.
  Denote the flag component and the underlying functor of $\varepsilon^{\flag}_{X}$ by $(\varepsilon^{\flag}_{X})_{\flagtag}:A_{\rmN_{\flag}X}\to A$ and $(\varepsilon^{\flag}_{X})_*:(\calC_{\rmN_{\flag}X})_*\to\calC_*$, respectively.

  (1)
  The triangle identity is $\rmN_{\flag}(\varepsilon^{\flag}_{X}) \eta^{\flag}_{\rmN_{\flag}X}\simeq\id_{\rmN_{\flag}X}$.
  Taking degree zero in the triangle identity above and using \cref{flagged-nerve.lem.value-of-flagged-nerve}, we obtain
  \begin{align}
    (\varepsilon^{\flag}_{X})_{\flagtag} (\eta^{\flag}_{\rmN_{\flag}X})_0
    \simeq\id_A.
  \end{align}
  The second factor identifies $A$ with $A_{\rmN_{\flag}X}$ by \cref{direct.lem.underlying-realization} (2).
  Under this identification, the preceding identity identifies $(\varepsilon^{\flag}_{X})_{\flagtag}$ with the identity of $A$.

  (2)
  By \cref{flagged-nerve.lem.value-of-flagged-nerve}, the second projection is $\pr_{X}:\iota_{\rev}^{\myop,*}\rmN_{\flag}X \to\rmN(\calC_*)$.
  The defining mate identity for $\theta_{\rmN_{\flag}X}$, naturality of the second projection, and the triangle identity in (1) give
  \begin{align}
    \rmN((\varepsilon^{\flag}_{X})_* \theta_{\rmN_{\flag}X}) \eta_{\iota_{\rev}^{\myop,*}\rmN_{\flag}X} 
    &\simeq
    \rmN((\varepsilon^{\flag}_{X})_*)\pr_{\ac_{\flag}\rmN_{\flag}X}
    \iota_{\rev}^{\myop,*}\eta^{\flag}_{\rmN_{\flag}X} \\
    &\simeq
    \pr_{X} \iota_{\rev}^{\myop,*}(\rmN_{\flag}(\varepsilon^{\flag}_{X})\eta^{\flag}_{\rmN_{\flag}X}) \\
    &\simeq
    \pr_{X}.
  \end{align}
  Taking adjuncts under the adjunction $\ac\dashv\rmN$ gives the natural equivalence
  \begin{align}
    (\varepsilon^{\flag}_{X})_* \theta_{\rmN_{\flag}X}
    \simeq
    \varepsilon_{\calC_*} \ac(\pr_{X}).
  \end{align}
  Here $\varepsilon_{\calC_*}$ is the counit of the adjunction $\ac\dashv\rmN$.
  By \cref{direct.lem.ordinary-completion} (3), applied to $\epsilon_{\calC}i:A\to\iota_0\calC_*$, the functor on the right is fully faithful with essential image $(\calC_*)_A$.
  Thus it factors as
  \begin{align}
    \ac(\iota_{\rev}^{\myop,*}\rmN_{\flag}X)
    \xrightarrow{\simeq}(\calC_*)_A
    \hookrightarrow\calC_*.
  \end{align}
  Since $\theta_{\rmN_{\flag}X}$ is an equivalence, $(\varepsilon^{\flag}_{X})_*$ is fully faithful with essential image $(\calC_*)_A$.

  (3)
  By \cref{target.lem.hermitian-flag-equivalences}, the flag anima and the underlying $(\infty,1)$-category jointly detect equivalences of Hermitian flags.
  By (1) and (2), $\varepsilon^{\flag}_{X}$ is therefore an equivalence if and only if the full inclusion $(\calC_*)_A\hookrightarrow\calC_*$ is essentially surjective.
  This holds if and only if $\epsilon_{\calC}i:A\to\iota_0\calC_*$ is essentially surjective, which is exactly the flagged condition on $X$.
\end{proof}

\begin{theorem}[Flagged nerve equivalence]\label{direct.thm.flagged-equivalence}
  The following statements hold.
  \begin{enumerate}
    \item The flagged nerve $\rmN_{\flag} : \HermFlag_{(\infty,1)} \to \SegRev(\An)$ in \cref{flagged-nerve.lem.segal} restricts to an equivalence of $(\infty,1)$-categories
    \begin{align}
      \rmN_{\flag}|_{\FlDagCat_{(\infty,1)}}:
      \FlDagCat_{(\infty,1)}
      \xrightarrow{\simeq}
      \SegRev(\An),
    \end{align}
    with inverse $\ac_{\flag}|_{\SegRev(\An)}$.
    \item For every $X \in \FlDagCat_{(\infty,1)}$ and every $Y\in\HermFlag_{(\infty,1)}$,
    the natural morphism
    \begin{align}
      \Map_{\HermFlag_{(\infty,1)}}(X,Y)
      \xrightarrow{\simeq}
      \Map_{\Fun(\prism_{\rev}^{\myop},\An)}(\rmN_{\flag}X,\rmN_{\flag}Y)
    \end{align}
    is an equivalence.
  \end{enumerate}
\end{theorem}

\begin{proof}
  (1)
  By \cref{direct.thm.unit} (2), $\ac_{\flag}$ takes dagger Segal anima to flagged dagger $(\infty,1)$-categories.
  Thus the adjunction restricts to these full subcategories.
  Both the unit and the restricted counit are equivalences by \cref{direct.thm.unit} (1) and \cref{direct.thm.counit} (3).

  (2)
  Since $X$ is flagged, $\varepsilon^{\flag}_{X}$ is an equivalence by \cref{direct.thm.counit} (3).
  Precomposition with $\varepsilon^{\flag}_{X}$ and the realization--nerve adjunction give
  \begin{align}
    \Map_{\HermFlag_{(\infty,1)}}(X,Y)
    &\xrightarrow[\simeq]{(\varepsilon^{\flag}_{X})^*}
    \Map_{\HermFlag_{(\infty,1)}}(\ac_{\flag}\rmN_{\flag}X,Y) \\
    &\simeq
    \Map_{\Fun(\prism_{\rev}^{\myop},\An)}(\rmN_{\flag}X,\rmN_{\flag}Y).
  \end{align}
  The triangle identity $\rmN_{\flag}(\varepsilon^{\flag}_{X})\eta^{\flag}_{\rmN_{\flag}X}\simeq\id_{\rmN_{\flag}X}$ identifies the composite with the morphism induced by $\rmN_{\flag}$.
\end{proof}

\subsection{Unitary completeness and coflagging}

We define unitary completeness as locality with respect to the walking unitary.
For the nerve of a Hermitian flag, this is exactly full faithfulness of the morphism into the Hermitian anima (\cref{direct.thm.unitary}).
The flagged nerve equivalence therefore restricts to an equivalence between dagger $(\infty,1)$-categories and unitarily complete dagger Segal anima (\cref{direct.cor.unitary-equivalence}).

\begin{construction}
  Define a reversal simplicial anima $E_{\rev} : \prism_{\rev}^{\myop} \to \An$ by $E_{\rev}([n])=\{0,1\}^{n+1}$, with every map of $\prism_{\rev}$ acting by precomposition on these functions on vertices. 
\end{construction}

\begin{notation}
  Let $e_{\rev}:\yo_{\rev}[0]\to E_{\rev}$ select the constant tuple $0$. 
\end{notation}

\begin{definition}\label{direct.def.unitary-complete}
  Let $W\in\Fun(\prism_{\rev}^{\myop},\An)$.
  We will say that it is \emph{unitarily complete} if
  \begin{align}
    e_{\rev}^*:\Map_{\Fun(\prism_{\rev}^{\myop},\An)}(E_{\rev},W)
    \to\Map_{\Fun(\prism_{\rev}^{\myop},\An)}(\yo_{\rev}[0],W)\simeq W_0
  \end{align}
  is an equivalence of anima.
\end{definition}

\begin{notation}
  We define some full subcategories of $\Fun(\prism_{\rev}^{\myop},\An)$.
  \begin{itemize}
    \item The full subcategory
    \begin{align}
      \calU_{\rev}(\An) 
      \subseteq 
      \Fun(\prism_{\rev}^{\myop},\An)
    \end{align}
    spanned by unitarily complete reversal simplicial anima.
    \item The full subcategory
    \begin{align}
      \SegRevU(\An)
      :=\SegRev(\An)\cap\calU_{\rev}(\An) 
      \subseteq \Fun(\prism_{\rev}^{\myop},\An)
    \end{align}
    spanned by unitarily complete dagger Segal anima.
  \end{itemize}
\end{notation}

\begin{notation}
  We fix some notational conventions.
  \begin{itemize}
    \item Let $I$ be the indiscrete groupoid on $\{0,1\}$, regarded as an anima.
    \item Define a Hermitian flag (indeed a dagger $(\infty,1)$-category)
    \begin{align}
      \mathfrak e 
      := 
      (\id : * \to \Herm(j^{hC_2}\triv(*)) \simeq *).
    \end{align}
    \item Define a Hermitian flag (indeed a flagged dagger $(\infty,1)$-category)
    \begin{align}
      \mathfrak u 
      := 
      (i_{\mathfrak u} : \{0,1\} \to \Herm(j^{hC_2}\triv(I))),
    \end{align}
    where $i_{\mathfrak u}$ is adjoint to the inclusion $j^{hC_2}\triv(\{0,1\}) \hookrightarrow j^{hC_2}\triv(I)$.
    \item Write 
    \begin{align}
      0 : \mathfrak e \to \mathfrak u
    \end{align}
    for a morphism of Hermitian flags selecting the object $0$ of $\{0,1\}$.
  \end{itemize}
\end{notation}

\begin{lemma}\label{direct.lem.e-and-u}
  For every $X=(i : A \to \Herm(\calC)) \in \HermFlag_{(\infty,1)}$, the following statements hold.
  \begin{enumerate}
    \item There is an equivalence of anima
    \begin{align}
      \Map_{\HermFlag_{(\infty,1)}}(\mathfrak e,X)
      \simeq A.
    \end{align}
    \item There is an equivalence of anima
    \begin{align}
      \Map_{\HermFlag_{(\infty,1)}}(\mathfrak u,X)
      \simeq A\times_{\Herm(\calC)}A.
    \end{align}
    \item Under the equivalences in (1) and (2), precomposition with $0:\mathfrak e\to\mathfrak u$,
    \begin{align}
      0^*:
      \Map_{\HermFlag_{(\infty,1)}}(\mathfrak u,X)
      \to
      \Map_{\HermFlag_{(\infty,1)}}(\mathfrak e,X),
    \end{align}
    corresponds to the first projection $p_0:A\times_{\Herm(\calC)}A\to A$.
  \end{enumerate}
\end{lemma}

\begin{proof}
  (1)
  The adjunction $j^{hC_2}\triv\dashv\Herm$ and \cref{target.prop.mapping-spaces} give
  \begin{align}
    \Map_{\HermFlag_{(\infty,1)}}(\mathfrak e,X)
    &\simeq
    \Map_{\An}(*,A)\times_{\Map_{\An}(*,\Herm(\calC))}\Hom_{\AInvCat_{(\infty,1)}}(j^{hC_2}\triv(*),\calC)
    \\
    &\simeq 
    A\times_{\Herm(\calC)}\Herm(\calC)
    \\
    &\simeq 
    A.
  \end{align}

  (2)
  The same adjunction and \cref{target.prop.mapping-spaces} give
  \begin{align}
    \Map_{\HermFlag_{(\infty,1)}}(\mathfrak u,X)
    &\simeq
    \Map_{\An}(\{0,1\},A)\times_{\Map_{\An}(\{0,1\},\Herm(\calC))}
    \Hom_{\AInvCat_{(\infty,1)}}(j^{hC_2}\triv(I),\calC)
    \\
    &\simeq A^2\times_{\Herm(\calC)^2}\Map_{\An}(I,\Herm(\calC))
    \\
    &\simeq 
    A^2\times_{\Herm(\calC)^2}\Herm(\calC) \\
    &\simeq 
    A\times_{\Herm(\calC)}A.
  \end{align}
  Here the constant-path morphism $\Herm(\calC)\to\Map_{\An}(I,\Herm(\calC))$ is an equivalence, since $I$ is contractible.

  (3)
  By naturality of \cref{target.prop.mapping-spaces}, precomposition with $0$ is induced by evaluation at $0$ on both $\Map_{\An}(\{0,1\},A)$ and $\Map_{\An}(I,\Herm(\calC))$.
  Thus the morphism on the pullbacks in (1) and (2) is
  \begin{align}
    A\times_{\Herm(\calC)}A
    &\simeq 
    A^2\times_{\Herm(\calC)^2}\Herm(\calC) \\
    &\to
    A\times_{\Herm(\calC)}\Herm(\calC) \\
    &\simeq 
    A,
  \end{align}
  induced by the first projections $A^2\to A$ and $\Herm(\calC)^2\to\Herm(\calC)$ and the identity of $\Herm(\calC)$.
  Thus this is the first projection $p_0:A\times_{\Herm(\calC)}A\to A$.
\end{proof}

\begin{proposition}\label{direct.lem.walking-unitary}
  There is an equivalence in $\Fun([1],\Fun(\prism_{\rev}^{\myop},\An))$:
  \begin{align}
    \rmN_{\flag}(0 : \mathfrak e \to \mathfrak u)
    \simeq
    (e_{\rev}:\yo_{\rev}[0]\to E_{\rev}).
  \end{align}
\end{proposition}

\begin{proof}
  The underlying $(\infty,1)$-category of $j^{hC_2}\triv(I)$ is $jI$, and the underlying flag of $\mathfrak u$ is the inclusion $\{0,1\} \hookrightarrow I$.
  By \cref{flagged-nerve.lem.value-of-flagged-nerve}, for every $[n]\in\prism_{\rev}$, there is a pullback square in $\An$:
  \begin{align}
    \begin{tikzpicture}[auto,baseline=(current bounding box.center)]
      \node (tl) at (0,1.5) {$\rmN_{\flag}(\mathfrak u)_{n}$};
      \node (tr) at (4,1.5) {$\iota_0\Fun([n],jI)$};
      \node (bl) at (0,0) {$\{0,1\}^{n+1}$};
      \node (br) at (4,0) {$I^{n+1}$};
      \draw[->] (tl) to (tr);
      \draw[->] (tl) to node[swap] {$p_n$} (bl);
      \draw[->] (tr) to node {$\scriptstyle(\ev_0,\ldots,\ev_n)$} (br);
      \draw[->] (bl) to (br);
    \end{tikzpicture}
  \end{align}
  Since $I$ is the indiscrete groupoid on $\{0,1\}$, vertex evaluation is an equivalence.
  Hence $p_n$ is an equivalence.

  For every order-preserving or order-reversing map $\beta:[m]\to[n]$, the flag component of $[\beta]_{\flag} : [m]_{\flag} \to [n]_{\flag}$ is its map of vertices.
  Thus we have 
  \begin{align}
    p_m \rmN_{\flag}(\mathfrak u)(\beta)
    \simeq
    \beta^* p_n,
    \quad \text{where} \quad
    \beta^*(a_0,\ldots,a_n)
    :=(a_{\beta(0)},\ldots,a_{\beta(m)}).
  \end{align}
  Therefore they give an equivalence in $\Fun(\prism_{\rev}^{\myop},\An)$:
  \begin{align}
    \rmN_{\flag}(\mathfrak u)\xrightarrow{\simeq}E_{\rev}.
  \end{align}

  The same argument for the terminal groupoid and its singleton flag gives an equivalence in $\Fun(\prism_{\rev}^{\myop},\An)$:
  \begin{align}
    \rmN_{\flag}(\mathfrak e)\simeq\yo_{\rev}[0].
  \end{align}

  The flag component of $0:\mathfrak e\to\mathfrak u$ selects $0\in\{0,1\}$.
  Hence the natural flag projections give a commutative square in $\Fun(\prism_{\rev}^{\myop},\An)$:
  \begin{align}
    \begin{tikzpicture}[auto,baseline=(current bounding box.center)]
      \node (tl) at (0,1.5) {$\rmN_{\flag}(\mathfrak e)$};
      \node (tr) at (4,1.5) {$\rmN_{\flag}(\mathfrak u)$};
      \node (bl) at (0,0) {$\yo_{\rev}[0]$};
      \node (br) at (4,0) {$E_{\rev}$};
      \draw[->] (tl) to node {$\rmN_{\flag}(0)$} (tr);
      \draw[->] (tl) to node[swap] {$\simeq$} (bl);
      \draw[->] (tr) to node {$\simeq$} (br);
      \draw[->] (bl) to node {$e_{\rev}$} (br);
    \end{tikzpicture}
  \end{align}
\end{proof}

\begin{proposition}\label{direct.thm.unitary}
  For every $X=(i : A \to \Herm(\calC)) \in \HermFlag_{(\infty,1)}$, the following conditions are equivalent.
  \begin{enumerate}
    \item The flagged nerve $\rmN_{\flag}X$ is unitarily complete; 
    $\rmN_{\flag}X \in \calU_{\rev}(\An)$.
    \item The Hermitian flag $i:A\to\Herm(\calC)$ is fully faithful;
    $X \in \coFlDagCat_{(\infty,1)}$.
  \end{enumerate}
\end{proposition}

\begin{proof}
  By \cref{direct.thm.flagged-equivalence} (2) and \cref{direct.lem.walking-unitary}, the flagged nerve induces the following commutative square in $\An$ whose vertical maps are equivalences:
  \begin{align}
    \begin{tikzpicture}[auto,baseline=(current bounding box.center)]
      \node (tl) at (0,1.5) {$\Map_{\Fun(\prism_{\rev}^{\myop},\An)}(E_{\rev},\rmN_{\flag}X)$};
      \node (tr) at (6,1.5) {$\Map_{\Fun(\prism_{\rev}^{\myop},\An)}(\yo_{\rev}[0],\rmN_{\flag}X)$};
      \node (bl) at (0,0) {$\Map_{\HermFlag_{(\infty,1)}}(\mathfrak u,X)$};
      \node (br) at (6,0) {$\Map_{\HermFlag_{(\infty,1)}}(\mathfrak e,X)$};
      \draw[->] (tl) to node {$e_{\rev}^{*}$} (tr);
      \draw[->] (bl) to node {$0^{*}$} (br);
      \draw[->] (bl) to node {$\simeq$} (tl);
      \draw[->] (br) to node[swap] {$\simeq$} (tr);
    \end{tikzpicture}
  \end{align}
  By \cref{direct.lem.e-and-u}, the lower horizontal map identifies with the first projection $p_0:A\times_{\Herm(\calC)}A\to A$.
  It is a retraction of the diagonal $\delta:A\to A\times_{\Herm(\calC)} A$ with $p_0\delta\simeq\id_A$.
  Then $p_{0}$ is an equivalence if and only if $\delta$ is an equivalence.
  The latter condition is equivalent to full faithfulness of $i:A\to\Herm(\calC)$.
\end{proof}

\begin{theorem}\label{direct.cor.unitary-equivalence}
  The equivalence of \cref{direct.thm.flagged-equivalence} (1) restricts to an equivalence of $(\infty,1)$-categories
  \begin{align}
    \rmN_{\flag}|_{\DagCatInf_{(\infty,1)}}
    : \DagCatInf_{(\infty,1)}
    \xrightarrow{\simeq}
    \SegRevU(\An).
  \end{align}
\end{theorem}

\begin{proof}
  Let $X=(i:A\to\Herm(\calC))\in\FlDagCat_{(\infty,1)}$.
  By \cref{flagged-nerve.lem.segal}, $\rmN_{\flag}X$ is already Segal.
  By \cref{direct.thm.unitary}, $X\in\DagCatInf_{(\infty,1)}$ if and only if $\rmN_{\flag}X \in \SegRevU(\An)$.
  Since $\DagCatInf_{(\infty,1)}=\FlDagCat_{(\infty,1)}\cap\coFlDagCat_{(\infty,1)}$, restricting $\rmN_{\flag}|_{\FlDagCat_{(\infty,1)}}$ to $\DagCatInf_{(\infty,1)}$ and restricting its codomain to $\SegRevU(\An)$ gives the desired equivalence.
\end{proof}

\begin{corollary}\label{direct.cor.coflagged-is-localization}
  The $(\infty,1)$-category $\coFlDagCat_{(\infty,1)}$ is a localization of $\HermFlag_{(\infty,1)}$ with respect to the morphism $\mathfrak e\to\mathfrak u$.
\end{corollary}

\begin{corollary}\label{direct.cor.unitary-inverse-image}
  There is an equivalence of $(\infty,1)$-categories
  \begin{align}
    \coFlDagCat_{(\infty,1)}
    \simeq
    \HermFlag_{(\infty,1)}
    \times_{\rmN_{\flag},\Fun(\prism_{\rev}^{\myop},\An),\inc}\calU_{\rev}(\An).
  \end{align}
\end{corollary}

\section{Univalentization and Hermitian completion}\label{flag.section}

In this section, we construct the reflection onto coflagged dagger $(\infty,1)$-categories and the coreflection onto flagged dagger $(\infty,1)$-categories.
Image factorization of the Hermitian flag gives the first, while restriction to its underlying essential image gives the second.
Their restrictions yield univalentization and Hermitian completion, and we compare these adjunctions with the nerve equivalences of Theorem~A (\cref{flag.thm.univalentization,flag.thm.hermitian-completion}).

\subsection{Univalentization as a left adjoint}

Image factorization replaces a Hermitian flag by the full subanima of Hermitian objects which it selects.
This defines a reflection onto coflagged dagger $(\infty,1)$-categories and restricts to univalentization of flagged ones (\cref{flag.prop.ambient-coflagged-reflection,flag.thm.univalentization}).
Under the flagged nerve equivalence, the resulting adjunction corresponds to unitary completion (\cref{flag.cor.unitary-completion}).

\begin{notation}
  For a morphism of anima $i:A\to H$, write
  \begin{align}
    A \xtwoheadrightarrow{e} \image(i) \xhookrightarrow{m} H
  \end{align}
  for its image factorization, where $e$ is essentially surjective and $m$ is fully faithful.
  Here $\image(i)$ is the full subanima of $H$ on the components meeting $i(A)$.
  This factorization is functorial by \cite[Example~5.2.8.16 and Proposition~5.2.8.17]{HTT}.
\end{notation}

\begin{proposition}\label{flag.prop.ambient-coflagged-reflection}
  The inclusion $\inc_{\co\flag} : \coFlDagCat_{(\infty,1)} \hookrightarrow \HermFlag_{(\infty,1)}$ admits a left adjoint
  \begin{align}
    L_{\co\flag}
    : \HermFlag_{(\infty,1)} \to \coFlDagCat_{(\infty,1)}
    : (i:A\to\Herm(\calC)) \mapsto (m:\image(i)\hookrightarrow\Herm(\calC)).
  \end{align}
\end{proposition}

\begin{proof}
  By \cite[Example~5.2.8.16 and Lemma~5.2.8.19]{HTT}, image factorization is a reflection of $\Fun([1],\An)$ onto its full subcategory of fully faithful morphisms:
  \begin{align}
    L_{\fftag} 
    : \Fun([1],\An) \rightleftarrows \Fun([1],\An)^{\fftag} : \inc.
  \end{align}
  This reflection preserves target evaluation.
  By \cref{target.lem.equivalences-of-flagged-and-coflagged}, we have 
  \begin{align}
    \coFlDagCat_{(\infty,1)}
    \simeq 
    \Fun([1],\An)^{\fftag} \times_{\ev_{1},\An,\Herm} \AInvCat_{(\infty,1)}.
  \end{align}
  Pulling back the adjunction $L_{\fftag} \dashv \inc$ along $\Herm:\AInvCat_{(\infty,1)}\to\An$ gives an adjunction
  \begin{align}
    L_{\co\flag}:\HermFlag_{(\infty,1)}
    \rightleftarrows\coFlDagCat_{(\infty,1)}
    :\inc_{\co\flag}.
  \end{align}
  Its unit is the image-factorization morphism on the flag anima and the identity on the underlying anti-involutive $(\infty,1)$-category, as asserted.
\end{proof}

\begin{theorem}[Univalentization; {\cite[Statement~2.5]{FerrerEtAl24}}]\label{flag.thm.univalentization}
  The adjunction $L_{\co\flag}\dashv\inc_{\co\flag}$ in \cref{flag.prop.ambient-coflagged-reflection} restricts to an adjunction
  \begin{align}
    L_{\co\flag}|_{\flag} 
    : \FlDagCat_{(\infty,1)} \rightleftarrows \DagCatInf_{(\infty,1)} 
    : \inc_{\co\flag}|_{\flag}.
  \end{align}
\end{theorem}

\begin{proof}
  By definition, the inclusion $\inc_{\co\flag}$ restricts to an inclusion 
  \begin{align}
    \inc_{\co\flag}|_{\flag}
    := \inc_{\co\flag}|_{\DagCatInf_{(\infty,1)}} 
    : \DagCatInf_{(\infty,1)} \to \FlDagCat_{(\infty,1)}.
  \end{align}
  For $X = (i : A \to \Herm(\calC)) \in \FlDagCat_{(\infty,1)}$, the composite
  \begin{align}
    A
    \xtwoheadrightarrow{e}
    \image(i)
    \xhookrightarrow{m}
    \Herm(\calC)
    \xrightarrow{\epsilon_{\calC}}
    \iota_0\calC_*
  \end{align}
  is essentially surjective.
  Hence $\epsilon_{\calC}m$ is also essentially surjective, so $L_{\co\flag}X$ is flagged.
  Thus $L_{\co\flag}X \in \DagCatInf_{(\infty,1)}$. 
\end{proof}

\begin{corollary}[Univalentization is unitary completion]\label{flag.cor.unitary-completion}
  Under the equivalences \cref{direct.thm.flagged-equivalence} (1) and \cref{direct.cor.unitary-equivalence}, the adjunction $L_{\co\flag}|_{\flag} \dashv \inc_{\co\flag}|_{\flag}$ corresponds to the localization at $e_{\rev}$:
  \begin{align}
    L_{\rmu}
    :\SegRev(\An) \rightleftarrows \SegRevU(\An)
    : \inc.
  \end{align}
\end{corollary}

\begin{corollary}
  The left adjoints
  \begin{align}
    L_{\co\flag} : \HermFlag_{(\infty,1)} \to \coFlDagCat_{(\infty,1)} 
    \quad \text{and} \quad 
    L_{\co\flag}|_{\flag} : \FlDagCat_{(\infty,1)} \to \DagCatInf_{(\infty,1)}
  \end{align}
  preserve small colimits.
\end{corollary}

\begin{proposition}
  The left adjoints
  \begin{align}
    L_{\co\flag} : \HermFlag_{(\infty,1)} \to \coFlDagCat_{(\infty,1)}
    \quad \text{and} \quad
    L_{\co\flag}|_{\flag} : \FlDagCat_{(\infty,1)} \to \DagCatInf_{(\infty,1)}
  \end{align}
  preserve small products.
\end{proposition}

\begin{proof}
  Let $\Lambda$ be a small set, and let $X_\lambda=(i_\lambda:A_\lambda\to\Herm(\calC_\lambda))$ be a family of Hermitian flags.
  Since $\Herm$ is a right adjoint, products in $\HermFlag_{(\infty,1)}$ are computed on the flag anima and the underlying anti-involutive $(\infty,1)$-categories.
  The flag of $\prod_{\Lambda}X_\lambda$ is therefore given by 
  \begin{align}
    \prod_{\Lambda}A_\lambda
    \xrightarrow{\prod_{\Lambda}i_\lambda}
    \prod_{\Lambda}\Herm(\calC_\lambda)
    \simeq
    \Herm(\prod_{\Lambda}\calC_\lambda).
  \end{align}

  For each $\lambda\in \Lambda$, consider the image factorization
  \begin{align}
    A_\lambda 
    \xtwoheadrightarrow{e_\lambda} 
    \image(i_\lambda) 
    \xhookrightarrow{m_\lambda} 
    \Herm(\calC_\lambda),
  \end{align}
  where $e_\lambda$ is essentially surjective and $m_\lambda$ is fully faithful.
  Then 
  \begin{align}
    \prod_{\Lambda}A_\lambda
    \xtwoheadrightarrow{\prod_{\Lambda}e_\lambda}
    \prod_{\Lambda}\image(i_\lambda)
    \xhookrightarrow{\prod_{\Lambda}m_\lambda}
    \prod_{\Lambda}\Herm(\calC_\lambda)
  \end{align}
  is also an image factorization.

  Essential surjectivity and full faithfulness are preserved by small products.
  Since $L_{\co\flag}$ does not change the underlying anti-involutive $(\infty,1)$-category, the formula in \cref{flag.prop.ambient-coflagged-reflection} gives an equivalence in $\coFlDagCat_{(\infty,1)}$:
  \begin{align}
    L_{\co\flag}(\prod_{\Lambda}X_\lambda)
    \xrightarrow{\simeq}
    \prod_{\Lambda}L_{\co\flag}(X_\lambda).
  \end{align}
  The empty product is the common terminal Hermitian flag $(\id_*:* \to\Herm(\oneDag)\simeq *)$, which is fixed by $L_{\co\flag}$.
  For a family in $\FlDagCat_{(\infty,1)}$, the same comparison gives an equivalence in $\DagCatInf_{(\infty,1)}$:
  \begin{align}
    L_{\co\flag}|_{\flag}(\prod_{\Lambda}X_\lambda)
    \xrightarrow{\simeq}
    \prod_{\Lambda}L_{\co\flag}|_{\flag}(X_\lambda).
  \end{align}
\end{proof}

\subsection{Hermitian completion as a right adjoint}

We restrict the underlying anti-involutive $(\infty,1)$-category to the full subcategory spanned by objects equivalent to those selected by the flag.
The resulting coreflection onto flagged dagger $(\infty,1)$-categories restricts to Hermitian completion on coflagged ones (\cref{flag.thm.hermitian-completion}).
We also identify this coreflection with the composite of the flagged nerve and realization (\cref{flag.prop.coreflection-realization}).

\begin{lemma}\label{flag.lem.anti-involutive-factorization}
  The following statements hold.
  \begin{enumerate}
    \item The morphisms in $\AInvCat_{(\infty,1)}$ whose underlying functors are essentially surjective and those whose underlying functors are fully faithful form a factorization system.
    \item Let $F:\calC\to\calD$ be a morphism in $\AInvCat_{(\infty,1)}$.
    The underlying $(\infty,1)$-category of $\image(F)$ is equivalent, over $\calD_*$, to the full subcategory spanned by the essential image of $F_*$.
  \end{enumerate}
\end{lemma}

\begin{proof}
  (1)
  Essentially surjective and fully faithful functors form a factorization system on $\Cat_{(\infty,1)}$.
  Write $\Fun([2],\Cat_{(\infty,1)})^{\estag,\fftag}$ for the full subcategory of composable pairs whose first arrow is essentially surjective and whose second arrow is fully faithful.
  By \cite[Proposition~5.2.8.17]{HTT}, the restriction along the inclusion $\{0,2\} \subseteq [2]$ induces an equivalence
  \begin{align}
    \Fun([2],\Cat_{(\infty,1)})^{\estag,\fftag}
    \xrightarrow{\simeq}
    \Fun([1],\Cat_{(\infty,1)}).
  \end{align}
  Since the opposite action preserves both classes, this lifts to an equivalence of $C_2$-categories with the action induced pointwise by $(-)^{\myop}$.
  
  Taking homotopy fixed points and applying \cref{target.lem.homotopy-fixed-point-formalism} (1), (7), we obtain an equivalence
  \begin{align}
    \Fun([2],\AInvCat_{(\infty,1)})^{\estag,\fftag}
    \xrightarrow{\simeq}
    \Fun([1],\AInvCat_{(\infty,1)}).
  \end{align}
  These classes contain all equivalences and are closed under composition.
  Thus these classes form the desired factorization system.

  (2)
  Evaluating its factorization at $*\in BC_2$ recovers the essentially surjective--fully faithful factorization of $F_*$, whose intermediate category is its full image.
\end{proof}

\begin{proposition}\label{flag.prop.ambient-flagged-coreflection}
  The inclusion $\inc_{\flag} : \FlDagCat_{(\infty,1)} \hookrightarrow \HermFlag_{(\infty,1)}$ admits a right adjoint
  \begin{align}
    R_{\flag}
    : \HermFlag_{(\infty,1)} \to \FlDagCat_{(\infty,1)}
    : (i : A \to \Herm(\calC)) \mapsto (\overline{i} : A \to \Herm(\calC_{S_X})).
  \end{align}
  Here for $X=(i:A\to\Herm(\calC))$,
  \begin{itemize}
    \item $S_X := \image(\pi_0(\epsilon_{\calC}i):\pi_0(A)\to\pi_0(\iota_0\calC_*))$;
    \item $\calC_{S_X}$ is the anti-involutive $(\infty,1)$-category whose underlying $(\infty,1)$-category $(\calC_{S_X})_*$ is the full subcategory of $\calC_*$ spanned by the objects whose equivalence classes belong to $S_X$;
    \item $\overline{i} : A \to \Herm(\calC_{S_X})$ is the induced morphism whose composite with $\Herm(\calC_{S_X}) \to \Herm(\calC)$ is equivalent to $i$.
    Its underlying flag $\epsilon_{\calC_{S_X}}\overline{i}:A\to\iota_0(\calC_{S_X})_*$ is essentially surjective.
  \end{itemize}
\end{proposition}

\begin{proof}
  By \cref{flag.lem.anti-involutive-factorization} and the dual of \cite[Lemma~5.2.8.19]{HTT}, image factorization is a coreflection onto the full subcategory of arrows whose underlying functors are essentially surjective:
  \begin{align}
    \inc:
    \Fun([1],\AInvCat_{(\infty,1)})^{\estag}
    \rightleftarrows
    \Fun([1],\AInvCat_{(\infty,1)})
    :R_{\estag}.
  \end{align}
  This coreflection preserves source evaluation, and its counit has the identity as its source component.
  Hence it is an adjunction over $\AInvCat_{(\infty,1)}$ via $\ev_0$.

  By \cref{target.lem.equivalences-of-flagged-and-coflagged}, we have 
  \begin{align}
    \HermFlag_{(\infty,1)}
    &\simeq
    \Fun([1],\AInvCat_{(\infty,1)})
    \times_{\ev_0,\AInvCat_{(\infty,1)},j^{hC_2}\triv}\An,\\
    \FlDagCat_{(\infty,1)}
    &\simeq
    \Fun([1],\AInvCat_{(\infty,1)})^{\estag}
    \times_{\ev_0,\AInvCat_{(\infty,1)},j^{hC_2}\triv}\An.
  \end{align}
  Pulling back the adjunction $\inc \dashv R_{\estag}$ along $j^{hC_2}\triv:\An\to\AInvCat_{(\infty,1)}$ gives an adjunction
  \begin{align}
    \inc_{\flag}:\FlDagCat_{(\infty,1)}
    \rightleftarrows\HermFlag_{(\infty,1)}:R_{\flag}.
  \end{align}
  Then the factorization of $i^{\sharp} : j^{hC_{2}}\triv(A) \to \calC$ is given by 
  \begin{align}
    j^{hC_2}\triv(A) 
    \xtwoheadrightarrow{\overline{i}^{\sharp}}
    \calC_{S_X} 
    \hookrightarrow 
    \calC.
  \end{align}
  By \cref{flag.lem.anti-involutive-factorization}, $\calC_{S_{X}}$ is the stated full subcategory.
  The first arrow has an essentially surjective underlying functor, so $\epsilon_{\calC_{S_X}}\overline{i} : A \to \iota_{0}(\calC_{S_X})_{*}$ is essentially surjective.
\end{proof}

\begin{theorem}[Hermitian completion; {\cite[Statement~2.7]{FerrerEtAl24}}]\label{flag.thm.hermitian-completion}
  The adjunction $\inc_{\flag} \dashv R_{\flag}$ in \cref{flag.prop.ambient-flagged-coreflection} restricts to an adjunction
  \begin{align}
    \inc_{\flag}|_{\co\flag}
    : \DagCatInf_{(\infty,1)} \rightleftarrows \coFlDagCat_{(\infty,1)}
    : R_{\flag}|_{\co\flag}.
  \end{align}
\end{theorem}

\begin{proof}
  Since $\DagCatInf_{(\infty,1)}=\FlDagCat_{(\infty,1)}\cap\coFlDagCat_{(\infty,1)}$, restrict the domain of $\inc_{\flag}$ to $\DagCatInf_{(\infty,1)}$ and restrict its codomain to $\coFlDagCat_{(\infty,1)}$.
  Denote the resulting inclusion by
  \begin{align}
    \inc_{\flag}|_{\co\flag}:
    \DagCatInf_{(\infty,1)}\hookrightarrow\coFlDagCat_{(\infty,1)}.
  \end{align}

  Let $X=(i:A\to\Herm(\calC))\in\coFlDagCat_{(\infty,1)}$.
  By \cref{flag.prop.ambient-flagged-coreflection}, the flag of $R_{\flag}X$ is the factorization $\overline{i} : A \to \Herm(\calC_{S_X})$.
  Since $(\calC_{S_X})_*\hookrightarrow\calC_*$ is fully faithful, \cref{target.lem.equivariant-full-core} shows that $\Herm(\calC_{S_X})\to\Herm(\calC)$ is fully faithful.
  Since the composite $i$ is also fully faithful, $\overline{i}$ is fully faithful.
  Hence $R_{\flag}X$ is coflagged.
\end{proof}

\begin{corollary}
  The right adjoints 
  \begin{align}
    R_{\flag}
    : \HermFlag_{(\infty,1)} \to \FlDagCat_{(\infty,1)}
    \quad \text{and} \quad 
    R_{\flag}|_{\co\flag}
    : \coFlDagCat_{(\infty,1)} \to \DagCatInf_{(\infty,1)}
  \end{align}
  preserve small limits.
\end{corollary}

\begin{proposition}\label{flag.prop.hermitian-completion-coproducts}
  The right adjoints
  \begin{align}
    R_{\flag}
    : \HermFlag_{(\infty,1)} \to \FlDagCat_{(\infty,1)}
    \quad \text{and} \quad
    R_{\flag}|_{\co\flag}
    : \coFlDagCat_{(\infty,1)} \to \DagCatInf_{(\infty,1)}
  \end{align}
  preserve small coproducts.
\end{proposition}

\begin{proof}
  Let $\Lambda$ be a small set, and let $X_\lambda=(i_\lambda:A_\lambda\to\Herm(\calC_\lambda))$ be a family of Hermitian flags.
  By \cref{properties.prop.anti-involutive} (2), the coproduct $\coprod_{\Lambda}\calC_\lambda$ has underlying $(\infty,1)$-category $\coprod_{\Lambda}(\calC_\lambda)_*$, and its anti-involution preserves each summand.
  Then there is an equivalence in $\Fun(BC_2,\An)$:
  \begin{align}
    \iota_0^{hC_2}(\coprod_{\Lambda}\calC_\lambda)
    \simeq
    \coprod_{\Lambda}\iota_0^{hC_2}(\calC_\lambda).
  \end{align}
  The summand projections give an equivariant map
  \begin{align}
    \iota_0^{hC_2}(\coprod_{\mu\in\Lambda}\calC_\mu)
    \longrightarrow\triv(\Lambda).
  \end{align}
  Since $BC_2$ is connected and $\Lambda$ is discrete, we have
  \begin{align}
    \triv(\Lambda)^{hC_2}
    \simeq\Map_{\An}(BC_2,\Lambda)
    \simeq\Lambda.
  \end{align}
  Homotopy fixed points preserve pullbacks, so for each $\lambda\in\Lambda$,
  \begin{align}
    \Herm(\coprod_{\mu\in\Lambda}\calC_\mu)\times_{\Lambda}\{\lambda\}
    &\simeq
    (\iota_0^{hC_2}(\coprod_{\mu\in\Lambda}\calC_\mu)
    \times_{\triv(\Lambda)}\triv(\{\lambda\}))^{hC_2}\\
    &\simeq\Herm(\calC_\lambda).
  \end{align}
  Consequently, the canonical comparison is an equivalence
  \begin{align}
    \coprod_{\Lambda}\Herm(\calC_\lambda)
    \xrightarrow{\simeq}
    \Herm(\coprod_{\Lambda}\calC_\lambda).
  \end{align}
  
  By \cref{target.prop.mapping-spaces}, the coproduct $\coprod_{\Lambda}X_\lambda$ is therefore computed on the flag anima and the underlying anti-involutive $(\infty,1)$-categories, with flag
  \begin{align}
    \coprod_{\Lambda}A_\lambda
    \xrightarrow{\coprod_{\Lambda}i_\lambda}
    \coprod_{\Lambda}\Herm(\calC_\lambda)
    \xrightarrow{\simeq}
    \Herm(\coprod_{\Lambda}\calC_\lambda).
  \end{align}
  Essential surjectivity and full faithfulness of morphisms of anima are preserved by small coproducts.
  
  Under the canonical identification $\pi_0(\iota_0(\coprod_{\Lambda}\calC_\lambda)_*) \simeq \coprod_{\Lambda}\pi_0(\iota_0(\calC_\lambda)_*)$, we have 
  \begin{align}
    S_{\coprod_{\Lambda}X_\lambda}
    =
    \coprod_{\Lambda}S_{X_\lambda}.
  \end{align}
  There are no morphisms between distinct summands.
  Hence there is an equivalence in $\AInvCat_{(\infty,1)}$:
  \begin{align}
    (\coprod_{\Lambda}\calC_\lambda)_{S_{\coprod_{\Lambda}X_\lambda}}
    \simeq
    \coprod_{\Lambda}(\calC_\lambda)_{S_{X_\lambda}},
  \end{align}
  compatibly with their inclusions into $\coprod_{\Lambda}\calC_\lambda$.
  Under this equivalence, the flag of $R_{\flag}(\coprod_{\Lambda}X_\lambda)$ is the coproduct of the flags of $R_{\flag}(X_\lambda)$.
  The formula in \cref{flag.prop.ambient-flagged-coreflection} therefore identifies the canonical comparison with an equivalence
  \begin{align}
    \coprod_{\Lambda}R_{\flag}(X_\lambda)
    \xrightarrow{\simeq}
    R_{\flag}(\coprod_{\Lambda}X_\lambda).
  \end{align}
  The empty coproduct is the common initial Hermitian flag $(\id_{\emptyset}:\emptyset\to\Herm(\emptyset)\simeq\emptyset)$, which is fixed by $R_{\flag}$.
  For a family in $\coFlDagCat_{(\infty,1)}$, the same comparison gives an equivalence in $\DagCatInf_{(\infty,1)}$:
  \begin{align}
    \coprod_{\Lambda}R_{\flag}|_{\co\flag}(X_\lambda)
    \xrightarrow{\simeq}
    R_{\flag}|_{\co\flag}(\coprod_{\Lambda}X_\lambda).
  \end{align}
\end{proof}

\begin{corollary}\label{direct.prop.unitary-local-adjunction}
  There is an adjunction
  \begin{align}
    L_{\co\flag}\ac_{\flag}:
    \SegRevU(\An) \rightleftarrows \coFlDagCat_{(\infty,1)}
    :\rmN_{\flag}.
  \end{align}
\end{corollary}

\begin{proof}
  For every $W\in\SegRevU(\An)$ and every $X\in\coFlDagCat_{(\infty,1)}$, we have 
  \begin{align}
    \Map_{\coFlDagCat_{(\infty,1)}}(L_{\co\flag}\ac_{\flag}W,X)
    &\simeq\Map_{\HermFlag_{(\infty,1)}}(\ac_{\flag}W,X) \\
    &\simeq\Map_{\Fun(\prism_{\rev}^{\myop},\An)}(W,\rmN_{\flag}X) \\
    &\simeq\Map_{\SegRevU(\An)}(W,\rmN_{\flag}X).
  \end{align}
\end{proof}

\begin{corollary}\label{flag.prop.coreflection-realization}
  There is a natural equivalence of functors
  \begin{align}
    \ac_{\flag}\rmN_{\flag}
    \simeq 
    R_{\flag} 
    :
    \HermFlag_{(\infty,1)} 
    \to 
    \FlDagCat_{(\infty,1)}.
  \end{align}
\end{corollary}

\begin{proof}
  For every $X\in\HermFlag_{(\infty,1)}$, \cref{flagged-nerve.lem.segal} and \cref{direct.thm.unit} (2) show that $\ac_{\flag}\rmN_{\flag}X$ is flagged.
  The adjunction of \cref{flag.prop.ambient-flagged-coreflection} gives the factorization of $\varepsilon^{\flag}_{X}$ as 
  \begin{align}
    \ac_{\flag}\rmN_{\flag}X
    \xrightarrow{\alpha_X}
    \inc_{\flag}R_{\flag}X
    \xrightarrow{\gamma_X}
    X.
  \end{align}
  Here $\gamma_X:\inc_{\flag}R_{\flag}X\to X$ is the counit of the adjunction in \cref{flag.prop.ambient-flagged-coreflection}.
  By \cref{direct.thm.counit} (1), (2), the composite $\varepsilon^{\flag}_{X} = \gamma_X \alpha_X$ is the identity on the flag anima, and its underlying functor $(\varepsilon^{\flag}_{X})_{*}$ is fully faithful with essential image the full subcategory determined by $S_X$.
  The same statements hold for $\gamma_X$ by \cref{flag.prop.ambient-flagged-coreflection}.
  It follows that $\alpha_X$ is an equivalence on flag anima.

  Its underlying functor $(\alpha_X)_{*}$ is fully faithful since both $(\gamma_X)_{*}$ and $(\gamma_X)_{*}(\alpha_X)_{*}$ are fully faithful.
  It is essentially surjective since those two functors have the same essential image and $(\gamma_X)_{*}$ is fully faithful.
  Thus $\alpha_X$ is an equivalence on underlying anti-involutive $(\infty,1)$-categories.

  Therefore $\alpha_X$ is an equivalence by \cref{target.lem.hermitian-flag-equivalences} and \cref{properties.prop.anti-involutive} (2).
\end{proof}

\subsection{Compatibility of adjunctions}

We compare the two orders of applying the reflection and coreflection.
The comparison is the canonical Beck--Chevalley transformation, and it is an equivalence (\cref{flag.lem.inclusion-square-comparison,properties.prop.commuting-corrections}).

\begin{proposition}\label{properties.prop.presentability}
  The following statements hold.
  \begin{enumerate}
    \item The $(\infty,1)$-categories $\HermFlag_{(\infty,1)}$, $\FlDagCat_{(\infty,1)}$, $\coFlDagCat_{(\infty,1)}$, and $\DagCatInf_{(\infty,1)}$ are presentable.
    \item The four inclusions $\inc_{\flag}$, $\inc_{\co\flag}$, $\inc_{\co\flag}|_{\flag}$, and $\inc_{\flag}|_{\co\flag}$ and their four adjoints $R_{\flag}$, $L_{\co\flag}$, $R_{\flag}|_{\co\flag}$, and $L_{\co\flag}|_{\flag}$ are accessible.
  \end{enumerate}
\end{proposition}

\begin{proof}
  (1)
  The presentability of $\HermFlag_{(\infty,1)}$ follows from \cref{target.lem.ambient-presentability}.
  By \cref{direct.cor.coflagged-is-localization}, $\coFlDagCat_{(\infty,1)}$ is a presentable accessible localization.

  The $(\infty,1)$-categories $\SegRev(\An)$ and $\SegRevU(\An)$ are the localizations of $\Fun(\prism_{\rev}^{\myop},\An)$ at the spine morphisms and at the spine morphisms together with the completeness morphism, respectively.
  They are presentable by \cite[Proposition~5.5.4.15]{HTT}.
  Hence \cref{direct.thm.flagged-equivalence} (1) and \cref{direct.cor.unitary-equivalence} give presentability of $\FlDagCat_{(\infty,1)}$ and $\DagCatInf_{(\infty,1)}$.

  (2)
  All eight functors occur in the adjunctions between presentable $(\infty,1)$-categories, so they are accessible by \cite[Proposition~5.4.7.7]{HTT}.
\end{proof}

\begin{remark}
  Neither
  \begin{align}
    L_{\co\flag} 
    : \HermFlag_{(\infty,1)} \to \coFlDagCat_{(\infty,1)} 
    \quad \text{nor} \quad 
    L_{\co\flag}|_{\flag} 
    : \FlDagCat_{(\infty,1)} \to \DagCatInf_{(\infty,1)}
  \end{align}
  preserves pullbacks:
  consider the following cospan in both $\HermFlag_{(\infty,1)}$ and $\FlDagCat_{(\infty,1)}$:
  \begin{align}
    \begin{tikzpicture}[auto,baseline=(current bounding box.center)]
      \node (tr) at (5,1.5) {$(\id_{*} : * \to \Herm(\oneDag))$};
      \node (bl) at (0,0) {$(\id_{*} : * \to \Herm(\oneDag))$};
      \node (br) at (5,0) {$(! : \{0,1\} \to \Herm(\oneDag))$};
      \draw[->] (tr) to node {$\displaystyle v_{1}$} (br);
      \draw[->] (bl) to node {$\displaystyle v_{0}$}  (br);
    \end{tikzpicture}
  \end{align}
  Here $!$ denotes the unique morphism to $\Herm(\oneDag)\simeq *$, and $v_r$ is the identity on $\oneDag$ and selects $r\in\{0,1\}$ on the flag anima.
  
  The pullback of this cospan in $\HermFlag_{(\infty,1)}$ is $(! : \emptyset \to \Herm(\oneDag))$.
  Then we have
  \begin{align}
    L_{\co\flag}(! : \emptyset \to \Herm(\oneDag))
    \simeq (! : \emptyset \to \Herm(\oneDag)).
  \end{align}
  However, we have
  \begin{align}
    L_{\co\flag}(\id_{*} : * \to \Herm(\oneDag))
    \simeq 
    L_{\co\flag}(! : \{0,1\} \to \Herm(\oneDag))
    &\simeq 
    (\id_{*} : * \to \Herm(\oneDag)), \\
    L_{\co\flag}(! : \emptyset \to \Herm(\oneDag))
    &\simeq 
    (! : \emptyset \to \Herm(\oneDag)).
  \end{align}
  So the pullback of the image of this cospan under $L_{\co\flag}$ in $\coFlDagCat_{(\infty,1)}$ is $(\id_{*} : * \to \Herm(\oneDag))$.
  The canonical comparison has flag component $\emptyset\to *$, which is not an equivalence.
  Therefore the above pullback is not preserved by $L_{\co\flag}$.

  By \cref{flag.prop.ambient-flagged-coreflection}, the pullback of the original cospan in $\FlDagCat_{(\infty,1)}$ is instead
  \begin{align}
    R_{\flag}(! : \emptyset \to \Herm(\oneDag))
    \simeq
    (! : \emptyset \to \Herm(\emptyset)).
  \end{align}
  The restriction $L_{\co\flag}|_{\flag}$ fixes this object, whereas the pullback of the cospan obtained by applying it is $(\id_* : * \to \Herm(\oneDag))$ in $\DagCatInf_{(\infty,1)}$.
  The canonical comparison again has flag component $\emptyset\to *$.
  Hence $L_{\co\flag}|_{\flag}$ does not preserve this pullback.
\end{remark}

\begin{remark}
  Neither
  \begin{align}
    R_{\flag} : \HermFlag_{(\infty,1)} \to \FlDagCat_{(\infty,1)}
    \quad \text{nor} \quad 
    R_{\flag}|_{\co\flag} : \coFlDagCat_{(\infty,1)} \to \DagCatInf_{(\infty,1)}
  \end{align}
  preserves pushouts:
  consider the following span in both $\HermFlag_{(\infty,1)}$ and $\coFlDagCat_{(\infty,1)}$:
  \begin{align}
    \begin{tikzpicture}[auto,baseline=(current bounding box.center)]
      \node (tl) at (0,1.5) {$(! : \emptyset \to \Herm(\oneDag))$};
      \node (tr) at (5,1.5) {$(\id_{*} : * \to \Herm(\oneDag))$};
      \node (bl) at (0,0) {$(\id_{*} : * \to \Herm(\oneDag))$};
      \draw[->] (tl) to node {$\displaystyle !$} (tr);
      \draw[->] (tl) to node {$\displaystyle !$}  (bl);
    \end{tikzpicture}
  \end{align}
  
  The pushout of this span in $\HermFlag_{(\infty,1)}$ is $(! : \{0,1\} \to \Herm(\oneDag))$.
  However, we have
  \begin{align}
    R_{\flag}(! : \emptyset \to \Herm(\oneDag))
    &\simeq 
    (! : \emptyset \to \Herm(\emptyset)), \\
    R_{\flag}(\id_{*} : * \to \Herm(\oneDag))
    &\simeq 
    (\id_{*} : * \to \Herm(\oneDag)), \\
    R_{\flag}(! : \{0,1\} \to \Herm(\oneDag))
    &\simeq (! : \{0,1\} \to \Herm(\oneDag)).
  \end{align}
  On the other hand, the pushout of the span obtained by applying $R_{\flag}$ is
  \begin{align}
    (\id_{\{0,1\}} : \{0,1\} \to \Herm(\oneDag \amalg \oneDag)).
  \end{align}
  The canonical comparison to the image of the original pushout has underlying fold functor $(\oneDag)_*\amalg(\oneDag)_*\to(\oneDag)_*$, which is not an equivalence.
  Therefore the above pushout is not preserved by $R_{\flag}$.

  By \cref{flag.prop.ambient-coflagged-reflection}, the pushout of the original span in $\coFlDagCat_{(\infty,1)}$ is instead
  \begin{align}
    L_{\co\flag}(! : \{0,1\} \to \Herm(\oneDag))
    \simeq
    (\id_* : * \to \Herm(\oneDag)).
  \end{align}
  The restriction $R_{\flag}|_{\co\flag}$ fixes this object, whereas the pushout of the span obtained by applying it is again $(\id_{\{0,1\}} : \{0,1\} \to \Herm(\oneDag\amalg\oneDag))$, now in $\DagCatInf_{(\infty,1)}$.
  The canonical comparison again has underlying fold functor $(\oneDag)_*\amalg(\oneDag)_*\to(\oneDag)_*$.
  Hence $R_{\flag}|_{\co\flag}$ does not preserve this pushout.
\end{remark}

\begin{proposition}\label{flag.lem.inclusion-square-comparison}
  The following statements hold.
  \begin{enumerate}
    \item The following square of inclusions commutes.
    \begin{align*}
      \begin{tikzpicture}[auto,baseline=(current bounding box.center)]
        \node (tl) at (0,1.5) {$\DagCatInf_{(\infty,1)}$};
        \node (tr) at (4,1.5) {$\FlDagCat_{(\infty,1)}$};
        \node (bl) at (0,0) {$\coFlDagCat_{(\infty,1)}$};
        \node (br) at (4,0) {$\HermFlag_{(\infty,1)}$};
        \draw[->] (tl) -- node {$\inc_{\co\flag}|_{\flag}$} (tr);
        \draw[->] (tl) -- node[swap] {$\inc_{\flag}|_{\co\flag}$} (bl);
        \draw[->] (tr) -- node {$\inc_{\flag}$} (br);
        \draw[->] (bl) -- node {$\inc_{\co\flag}$} (br);
      \end{tikzpicture}
    \end{align*}
    \item The canonical restriction equivalence
    \begin{align}
      \theta:
      R_{\flag}\inc_{\co\flag}
      \xrightarrow{\simeq}
      (\inc_{\co\flag}|_{\flag})(R_{\flag}|_{\co\flag})
    \end{align}
    has a left Beck--Chevalley mate
    \begin{align*}
      c:(L_{\co\flag}|_{\flag})R_{\flag}
      \to
      (R_{\flag}|_{\co\flag})L_{\co\flag}
      :\HermFlag_{(\infty,1)}\to\DagCatInf_{(\infty,1)}.
    \end{align*}
    Let $\eta:\id_{\HermFlag_{(\infty,1)}}\to\inc_{\co\flag}L_{\co\flag}$ be the unit.
    For every $X\in\HermFlag_{(\infty,1)}$, the component $c_X$ is adjoint, under the adjunction $L_{\co\flag}|_{\flag}\dashv\inc_{\co\flag}|_{\flag}$, to
    \begin{align*}
      R_{\flag}X
      \xrightarrow{R_{\flag}(\eta_X)}
      R_{\flag}\inc_{\co\flag}L_{\co\flag}X
      \xrightarrow{\theta_{L_{\co\flag}X}}
      (\inc_{\co\flag}|_{\flag})(R_{\flag}|_{\co\flag})L_{\co\flag}X.
    \end{align*}
  \end{enumerate}
\end{proposition}

\begin{proof}
  (1) is clear.

  (2)
  By \cref{flag.thm.hermitian-completion}, $R_{\flag}|_{\co\flag}$ is the restriction of $R_{\flag}$ to $\coFlDagCat_{(\infty,1)}$, giving the canonical natural equivalence $\theta$.
  The horizontal adjunctions are $L_{\co\flag}\dashv\inc_{\co\flag}$ and $L_{\co\flag}|_{\flag}\dashv\inc_{\co\flag}|_{\flag}$.
  Define $c$ to be the left Beck--Chevalley mate of $\theta$ under these adjunctions.
  If $\varepsilon'$ denotes the counit of the second adjunction, the unit--counit formula for mates gives
  \begin{align}
    c_X
    =
    \varepsilon'_{(R_{\flag}|_{\co\flag})L_{\co\flag}X}
    (L_{\co\flag}|_{\flag})
    (\theta_{L_{\co\flag}X} R_{\flag}(\eta_X)).
  \end{align}
  The triangle identities identify its adjunct with the stated composite.
  This construction is natural in $X$.
\end{proof}

\begin{proposition}\label{properties.prop.commuting-corrections}
  The Beck--Chevalley transformation $c$ of \cref{flag.lem.inclusion-square-comparison} (2) is a natural equivalence
  \begin{align}
    (L_{\co\flag}|_{\flag})R_{\flag}
    \xrightarrow[\simeq]{c}
    (R_{\flag}|_{\co\flag})L_{\co\flag}
    : \HermFlag_{(\infty,1)}
    \to
    \DagCatInf_{(\infty,1)}.
  \end{align}
\end{proposition}

\begin{proof}
  For $X=(i:A\to\Herm(\calC)) \in \HermFlag_{(\infty,1)}$, factor $i$ as $A\xtwoheadrightarrow{e}\image(i)\xhookrightarrow{m}\Herm(\calC)$.
  Since $e$ is essentially surjective, we have 
  \begin{align}
    \image(\pi_0(\epsilon_{\calC}i))
    =\image(\pi_0(\epsilon_{\calC}m))=S_X.
  \end{align}

  By \cref{flag.prop.ambient-flagged-coreflection}, the counits at $X$ and $\inc_{\co\flag}L_{\co\flag}X$ have underlying functors which are full inclusions with essential image $(\calC_*)_{S_X}$.
  Since $\eta_X$ has the identity as its underlying functor, naturality of the counit shows that the underlying functor of $R_{\flag}(\eta_X)$ is a functor between these full subcategories over $\calC_*$.
  It is therefore an equivalence.
  
  The unit of $L_{\co\flag}|_{\flag}\dashv\inc_{\co\flag}|_{\flag}$ also has the identity as its underlying functor.
  Thus the mate characterization in \cref{flag.lem.inclusion-square-comparison} (2) shows that the underlying functor of $c_X$ is an equivalence.
  Write $\gamma_X:\calC_{S_X} \hookrightarrow \calC$ for the anti-involutive $(\infty,1)$-category component of the counit at $X$, and let $\bar\imath:A\to\Herm(\calC_{S_X})$ be the flag of $R_{\flag}X$.
  By \cref{target.lem.equivariant-full-core}, $\Herm(\gamma_X)$ is fully faithful.
  Since $i\simeq\Herm(\gamma_X)\bar\imath$, stability of image factorization under pullback gives
  \begin{align}
    \image(\bar\imath)
    &\simeq
    \Herm(\calC_{S_X})\times_{\Herm(\calC)}\image(i) \\
    &\simeq
    \image(i).
  \end{align}
  The last equivalence holds since every component of $\image(i)$ has an underlying object in $S_X$.

  The flag component of $R_{\flag}(\eta_X)$ is $e$, whereas the restricted unit at $R_{\flag}X$ has flag component $A\to\image(\bar\imath)$.
  Hence the mate characterization identifies the flag component of $c_X$ with the displayed comparison of image factorizations.
  Its composite with $A\to\image(\bar\imath)$ is $e$.
  Thus it is an equivalence as well.
  By \cref{target.lem.hermitian-flag-equivalences}, every component of the already defined natural transformation $c$ is an equivalence.
  Hence $c$ is an equivalence in $\Fun(\HermFlag_{(\infty,1)},\DagCatInf_{(\infty,1)})$.
\end{proof}

\section{Model-categorical presentations}\label{presentation.section}

In this section, we identify the model categories of Parts~I and~II as presentations of $\DagCatInf_{(\infty,1)}$ (\cref{presentation.cor.all-models}).
We first compare the completeness generator of Part~II with the walking unitary used above.
We then identify the underlying $(\infty,1)$-category of the localized injective model with unitarily complete dagger Segal anima.

\subsection{The point-set completeness morphism}

We compare the completeness generator of Part~II with the map $e_{\rev}:\yo_{\rev}[0]\to E_{\rev}$.
The comparison identifies the walking unitary diagrams, including the selected vertex, after objectwise localization from simplicial sets to anima (\cref{unitary.lem.part-ii-coordinate-walking-unitary}).

\begin{remark}
  Recall that the $(\infty,1)$-category $\An$ is the underlying $(\infty,1)$-category of the Kan-Quillen model structure $\sSet_{\KQ}$:
  \begin{align}
    \ell
    : N(\sSet_{\KQ})
    \to N(\sSet_{\KQ})[\{\text{weak homotopy equivalences}\}^{-1}]
    \simeq \An.
  \end{align}

  By \cite{Aka26-2}, $\Fun(\prism_{\rev}^{\myop},\sSet)_{\inj}$ denotes the injective model structure induced by $\sSet_{\KQ}$.
  We denote the class of weak equivalences in this model structure by $W_{\rev}$.

  By \cite[Proposition 1.3.4.25]{HA}, we have the localization
  \begin{align}
    \ell_{\rev}
    : N(\Fun(\prism_{\rev}^{\myop},\sSet)_{\inj})
    \to N(\Fun(\prism_{\rev}^{\myop},\sSet)_{\inj})[W_{\rev}^{-1}]
    \simeq \Fun(\prism_{\rev}^{\myop},\An).
  \end{align}
\end{remark}

\begin{notation}\label{presentation.not.point-set}
  We fix some notational conventions.
  \begin{itemize}
    \item Let $\delta:\Set\to\sSet$ denote the discrete-simplicial-set functor and also its pointwise extension to presheaves $\delta : \sSet^{\dag} \to \sSpace^{\dag}$.
    \item Write $\Delta^n_{\rev}:=\delta\Hom_{\prism_{\rev}}(-,[n])$ and let $\Sp[n]_{\rev}$ be the colimit of the consecutive representable edges glued along their shared vertices.
    \item Write $s_{n} : \Sp[n]_{\rev} \hookrightarrow \Delta^n_{\rev}$ for the canonical inclusion.
    \item Let $\bbI_{\dag}^{\str}$ be the ordinary indiscrete groupoid on $\{0,1\}$, with dagger given by inversion, regarded as a dagger simplicial category with discrete mapping spaces.
    \item Using the dagger coherent nerve $\rmN_{\dag}$, put
    \begin{align}
      E_{\rev}^{\pstag}
      :=
      \delta \rmN_{\dag}(\bbI_{\dag}^{\str})
      \quad \text{and} \quad
      e_{\rev}^{\pstag}
      :\Delta^0_{\rev}
      \to 
      E_{\rev}^{\pstag},
    \end{align}
    where $e_{\rev}^{\pstag}$ selects the object $0$.
    This is the completeness generator used in Part II.
  \end{itemize}
\end{notation}

\begin{lemma}\label{unitary.lem.part-ii-theta}
  There is a natural equivalence of reversal simplicial spaces
  \begin{align}
    \vartheta^{\pstag}
    : \delta(\{0,1\}^{\bullet+1})
    \xrightarrow{\cong}
    E_{\rev}^{\pstag}.
  \end{align}
\end{lemma}

\begin{proof}
  The mapping simplicial sets of $\bbI_{\dag}^{\str}$ are discrete, so its coherent nerve is its ordinary nerve.
  Each sequence $(x_0,\ldots,x_n)$ determines a unique chain $x_0\to\cdots\to x_n$.
  Thus there is a bijection
  \begin{align}
    \vartheta_{n}
    : \{0,1\}^{n+1} 
    \xrightarrow{\cong} 
    (N_{\dag}(\bbI_{\dag}^{\str}))_{n}
    : (x_0,\ldots,x_n)
    \mapsto 
    (x_0\to\cdots\to x_n).
  \end{align}

  We check that it is compatible with order-preserving and order-reversing maps.
  For an order-preserving map $\alpha:[m]\to[n]$, restriction sends $(x_0\to\cdots\to x_n)$ to the sequence $(x_{\alpha(0)},\ldots,x_{\alpha(m)})$.
  Thus we have 
  \begin{align}
    \alpha^{*}\vartheta_{n}(x_0,\ldots,x_n)
    = \vartheta_{m}(x_{\alpha(0)},\ldots,x_{\alpha(m)}).
  \end{align}
  Since a reversal map applies the dagger and reverses the chain, $\rho_n^*(x_0,\ldots,x_n)=(x_n,\ldots,x_0)$.
  If $\beta:[m]\to[n]$ is order reversing, then $\rho_n\beta$ is order preserving and
  \begin{align}
    \beta^{*}(x_0,\ldots,x_n)
    &=
    (\rho_n\beta)^{*}(x_n,\ldots,x_0)\\
    &=
    (x_{n-\rho_n\beta(0)},\cdots,x_{n-\rho_n\beta(m)})\\
    &=
    (x_{\beta(0)},\cdots,x_{\beta(m)}). 
  \end{align}
  For a constant map with value $r$, both formulas give $(x_r,\ldots,x_r)$.
  Thus these bijections commute with every arrow of $\prism_{\rev}$.
  Therefore we obtain an equivalence of reversal simplicial sets
  \begin{align}
    \vartheta
    : \{0,1\}^{\bullet+1} 
    \xrightarrow{\cong}
    N_{\dag}(\bbI_{\dag}^{\str})
  \end{align}
  Applying $\delta$ gives the desired equivalence of reversal simplicial spaces.
\end{proof}

\begin{corollary}\label{unitary.lem.part-ii-coordinate-walking-unitary}
  There is an equivalence in $\Fun([1],\Fun(\prism_{\rev}^{\myop},\An))$:
  \begin{align}
    \ell_{\rev}^{[1]}(e_{\rev}^{\pstag}:\Delta^0_{\rev}\to E_{\rev}^{\pstag})
    \simeq
    (e_{\rev}:\yo_{\rev}[0]\to E_{\rev}).
  \end{align}
\end{corollary}

\begin{proof}
  Since $[0]$ is terminal in $\prism_{\rev}$, the simplicial set $\Delta^0_{\rev}([n])$ consists of a single point for every $n\geq0$.
  In degree $n$, the morphism $e_{\rev}^{\pstag}$ sends this point to the chain $(0\xrightarrow{\id_0}0\xrightarrow{\id_0}\cdots\xrightarrow{\id_0}0)$.
  Thus the isomorphism $\vartheta^{\pstag}$ of \cref{unitary.lem.part-ii-theta} fits into the commutative square
  \begin{align}
    \begin{tikzpicture}[auto,baseline=(current bounding box.center)]
      \node (tl) at (0,1.5) {$\Delta^0_{\rev}$};
      \node (tr) at (4,1.5) {$\delta(\{0,1\}^{\bullet+1})$};
      \node (bl) at (0,0) {$\Delta^0_{\rev}$};
      \node (br) at (4,0) {$E_{\rev}^{\pstag}$};
      \draw[->] (tl) to node {$\scriptstyle *\mapsto(0,\ldots,0)$} (tr);
      \draw[->] (tl) to node[swap] {$\scriptstyle\id$} (bl);
      \draw[->] (tr) to node {$\scriptstyle\vartheta^{\pstag}$} node[swap] {$\scriptstyle\cong$} (br);
      \draw[->] (bl) to node {$\scriptstyle e_{\rev}^{\pstag}$} (br);
    \end{tikzpicture}
  \end{align}
  This defines an equivalence of arrows in $\Fun(\prism_{\rev}^{\myop},\sSet)$.

  For a set $S$, the anima $\ell(\delta S)$ is naturally equivalent to $S$, regarded as a discrete anima.
  Therefore, in degree $n$, we have
  \begin{align}
    \ell_{\rev}(\delta(\{0,1\}^{\bullet+1}))_{n}
    \simeq\{0,1\}^{n+1}.
  \end{align}
  For every map $\beta:[m]\to[n]$ of $\prism_{\rev}$, the induced morphism is
  \begin{align}
    \{0,1\}^{n+1}\to\{0,1\}^{m+1} 
    :(x_0,\ldots,x_n) \mapsto (x_{\beta(0)},\ldots,x_{\beta(m)}).
  \end{align}
  These are precisely the structure maps of $E_{\rev}$.
  Similarly, the definition $\Delta^0_{\rev}=\delta\Hom_{\prism_{\rev}}(-,[0])$ gives
  \begin{align}
    \ell_{\rev}(\Delta^0_{\rev})_{n}
    \simeq\Hom_{\prism_{\rev}}([n],[0])
    =\yo_{\rev}[0]_{n}.
  \end{align}
  These identifications are natural in $[n]$, so they give equivalences
  \begin{align}
    \ell_{\rev}(\Delta^0_{\rev})\simeq\yo_{\rev}[0]
    \quad \text{and} \quad
    \ell_{\rev}(\delta(\{0,1\}^{\bullet+1}))\simeq E_{\rev}.
  \end{align}
  Under these equivalences, the localization of the upper horizontal arrow is $e_{\rev}$.
  Applying the functor on arrow $(\infty,1)$-categories induced by $\ell_{\rev}$ to the above square therefore gives the desired equivalence in $\Fun([1],\Fun(\prism_{\rev}^{\myop},\An))$.
\end{proof}

\subsection{Localization and the models of Parts I and II}

We localize the injective model structure at the spine maps and the point-set completeness morphism.
Its underlying $(\infty,1)$-category is $\SegRevU(\An)$, so the flagged nerve identifies it with $\DagCatInf_{(\infty,1)}$ (\cref{presentation.prop.localized-presentation,presentation.prop.reduction}).

\begin{notation}\label{presentation.rem.localization-locality}
  Put
  \begin{align}
    S_{\rev}
    :=
    \{s_n:\Sp[n]_{\rev}\to\Delta^n_{\rev} \mid n\geq2\}
    \cup
    \{e_{\rev}^{\pstag}:\Delta^0_{\rev}\to E_{\rev}^{\pstag}\}.
  \end{align}
\end{notation}

\begin{proposition}\label{presentation.prop.localized-presentation}
  The following statements hold.
  \begin{enumerate}
    \item The fibrant objects of $L_{S_{\rev}}\Fun(\prism_{\rev}^{\myop},\sSet)_{\inj}$ are exactly the injectively fibrant reversal simplicial spaces whose associated reversal simplicial anima is a unitarily complete dagger Segal anima.
    \item The localized model category $L_{S_{\rev}}\Fun(\prism_{\rev}^{\myop},\sSet)_{\inj}$ presents $\SegRevU(\An)$:
    there is an equivalence of $(\infty,1)$-categories
    \begin{align}
      (L_{S_{\rev}}\Fun(\prism_{\rev}^{\myop},\sSet)_{\inj})_{(\infty,1)}
      \simeq
      \SegRevU(\An).
    \end{align}
  \end{enumerate}
\end{proposition}

\begin{proof}
  (1)
  The injective model structure $\Fun(\prism_{\rev}^{\myop},\sSet)_{\inj}$ is left proper, combinatorial, and simplicial;
  every object is cofibrant.
  For an injectively fibrant $W \in \Fun(\prism_{\rev}^{\myop},\sSet)_{\inj}$, locality is detected by $\ell_{\rev}$ under the presentation equivalence of \cite[Proposition~1.3.4.25]{HA}.

  Each vertex inclusion $\Delta^0_{\rev} \to \Delta^1_{\rev}$ in $\Sp[n]_{\rev}$ is a levelwise monomorphism.
  The iterated pushout therefore computes the corresponding homotopy colimit.
  Hence we have 
  \begin{align}
    \Map(\ell_{\rev}(\Delta^n_{\rev}),\ell_{\rev}(W))
    &\simeq
    \ell_{\rev}(W)_n,\\
    \Map(\ell_{\rev}(\Sp[n]_{\rev}),\ell_{\rev}(W))
    &\simeq
    \ell_{\rev}(W)_1 \times_{\ell_{\rev}(W)_0} \cdots \times_{\ell_{\rev}(W)_0} \ell_{\rev}(W)_1.
  \end{align}
  Consequently, locality for $s_{n}$ is precisely the Segal condition.
  By \cref{unitary.lem.part-ii-coordinate-walking-unitary}, locality for $e_{\rev}^{\pstag}$ is unitary completeness.

  (2)
  By \cite[Propositions~A.3.7.3 and A.3.7.8]{HTT}, we have 
  \begin{align}
    (L_{S_{\rev}}\Fun(\prism_{\rev}^{\myop},\sSet)_{\inj})_{(\infty,1)}
    &\simeq
    \Fun(\prism_{\rev}^{\myop},\An)_{\ell_{\rev}^{[1]}(S_{\rev})\text{-local}} \\
    &\simeq
    \SegRevU(\An).
  \end{align}
\end{proof}

\begin{corollary}\label{presentation.prop.reduction}
  There is an equivalence of $(\infty,1)$-categories
  \begin{align}
    (L_{S_{\rev}}\Fun(\prism_{\rev}^{\myop},\sSet)_{\inj})_{(\infty,1)}
    \simeq
    \DagCatInf_{(\infty,1)}.
  \end{align}
\end{corollary}

\begin{proof}
  Compose \cref{presentation.prop.localized-presentation} with \cref{direct.cor.unitary-equivalence}.
\end{proof}

\begin{corollary}\label{presentation.cor.all-models}
  The dagger Bergner, dagger Joyal, and dagger Rezk model categories of Parts I and II all present $\DagCatInf_{(\infty,1)}$.
\end{corollary}

\begin{proof}
  The Quillen equivalences are given by \cite[Theorem~3.5.6]{Aka26} and \cite[Theorem~6.2.8 and Corollary~6.2.9]{Aka26-2}.
  Thus they induce equivalences of underlying $(\infty,1)$-categories by \cite[Lemma~1.3.4.21]{HA}.
  Apply \cref{presentation.prop.reduction}.
\end{proof}

\appendix
\section{Unitary equivalences and the unitary Dwyer--Kan criterion}\label{unitary.section}

In this appendix, we describe unitary equivalences using the anima of ordinary equivalences and relative homotopy fixed points.
We prove that this construction agrees with mapping out of the walking unitary (\cref{direct.thm.coherent-unitary-representability}).
It then gives an intrinsic criterion for equivalences after unitary completion (\cref{presentation.thm.unitary-dk}).

\subsection{Unitary equivalences}\label{direct.subsection.coherent-unitaries}

We define the anima of unitary equivalences in a dagger Segal anima $W$ and identify it with the anima of morphisms $E_{\rev}\to W$ (\cref{direct.thm.coherent-unitary-representability}).

\begin{notation}
  Let $W\in\SegRev(\An)$.
  For $x,y\in W_0$, put
  \begin{align}
    \Map_W(x,y):=\fib_{(x,y)}(W_1\xrightarrow{(\ev_{0},\ev_{1})}W_0\times W_0).
  \end{align}
\end{notation}

\begin{definition}[{\cite[Proposition~10.1]{Rez10}}]\label{direct.lem.equivalences-pullback}
  Let $V:\prism^{\myop}\to\An$ be a Segal anima.
  Define the anima $V_1^{\eqtag}$ of \emph{equivalences} in $V$ by the following pullback square in $\An$:
  \begin{align}
    \begin{tikzpicture}[auto,baseline=(current bounding box.center)]
      \node (tl) at (0,1.5) {$V_1^{\eqtag}$};
      \node (tm) at (4,1.5) {$V_3$};
      \node (bl) at (0,0) {$V_0\times V_0$};
      \node (bm) at (4,0) {$V_1\times V_1$};
      \draw[->] (tl) to (tm);
      \draw[->] (tl) to (bl);
      \draw[->] (tm) to node {$\scriptstyle(d_1d_3,d_1d_0)$} (bm);
      \draw[->] (bl) to node {$\scriptstyle s_0\times s_0$} (bm);
    \end{tikzpicture}
  \end{align}
  The morphism $V_1^{\eqtag}\to V_1$ is the composite of the projection to $V_3$ with $d_0d_3:V_3\to V_1$, which evaluates at the edge $\{1,2\}\subset[3]$.
\end{definition}

\begin{lemma}\label{direct.remark.rezk-walking-equivalence}
  Let $V:\prism^{\myop}\to\An$ be a Segal anima.
  There is a natural equivalence of anima over $V_{1}$:
  \begin{align}
    \Map_{\Fun(\prism^{\myop},\An)}(\iota_{\rev}^{\myop,*}E_{\rev},V)
    \simeq 
    V_1^{\eqtag}.
  \end{align}
\end{lemma}

\begin{proof}
  The underlying simplicial anima $\iota_{\rev}^{\myop,*}E_{\rev}$ is the discrete nerve of $I$.
  \Cite[Theorem~6.2]{Rez01} gives the desired equivalence over $V_1$.
\end{proof}

\begin{construction}\label{direct.construction.equivalence-action}
  Let $W\in\SegRev(\An)$ and write $\ac_{\flag}W=(i_W:A_W\to\Herm(\calC_W))$.
  Define a morphism of $C_{2}$-anima 
  \begin{align}
    \widetilde i_W:
    \triv(W_0)\to \iota_0^{hC_2}\calC_W
  \end{align}
  to be the adjunct of the composite $i_W(\eta^{\flag}_W)_0 : W_0 \to \Herm(\calC_W)$ under the adjunction $\triv\dashv(-)^{hC_2}$.

  Define a $C_2$-anima $W_1^{\eqtag,C_2} : BC_{2} \to \An$ by the following pullback in $\Fun(BC_{2},\An)$:
  \begin{align}
    W_1^{\eqtag,C_2}
    :=
    \triv(W_0\times W_0)
    \times_{(\widetilde i_W)^2, (\iota_0^{hC_2}\calC_W)^2, (\ev_{0},\ev_{1})}
    \Fun^{\conj}(\triv([1]),\iota_0^{hC_2}\calC_W).
  \end{align}
\end{construction}

\begin{lemma}\label{direct.lem.equivalence-action}
  Let $W\in\SegRev(\An)$ and write $\ac_{\flag}W=(i_W:A_W\to\Herm(\calC_W))$.
  The unit $\eta^{\flag}_W : W \to \rmN_{\flag}\ac_{\flag}W$ induces a natural equivalence of anima over $W_0\times W_0$:
  \begin{align}
    (\iota_{\rev}^{\myop,*}W)_1^{\eqtag}
    \xrightarrow{\simeq}
    (W_1^{\eqtag,C_2})_*.
  \end{align}
\end{lemma}

\begin{proof}
  By \cref{direct.thm.unit} (1) and the degree-one pullback of \cref{flagged-nerve.lem.value-of-flagged-nerve}, the composite
  \begin{align}
    \iota_{\rev}^{\myop,*}W
    \xrightarrow{\iota_{\rev}^{\myop,*}\eta_W^{\flag}}
    \iota_{\rev}^{\myop,*}\rmN_{\flag}\ac_{\flag}W
    \xrightarrow{\pr_{\ac_{\flag}W}}
    \rmN(\calC_W)_*
  \end{align}
  is fully faithful and has degree-zero morphism $\epsilon_{\calC_W}i_W(\eta_W^{\flag})_0$.
  It therefore preserves and reflects equivalences.
  Restricting the degree-one pullback to equivalences gives
  \begin{align}
    (\iota_{\rev}^{\myop,*}W)_1^{\eqtag}
    &\simeq
    (W_0\times W_0)
    \times_{(\epsilon_{\calC_W}i_W(\eta_W^{\flag})_0)^2,(\iota_0(\calC_W)_*)^2,(\ev_0,\ev_1)}
    \Fun([1],\iota_0(\calC_W)_*).
  \end{align}
  By \cref{target.lem.homotopy-fixed-point-formalism} (6), evaluation at $*\in BC_2$ preserves the pullback in \cref{direct.construction.equivalence-action}.
  Since $(\widetilde i_W)_*\simeq\epsilon_{\calC_W}i_W(\eta_W^{\flag})_0$, the same pullback computes $(W_1^{\eqtag,C_2})_*$.
  This gives the desired equivalence, naturally in $W$ and over $W_0\times W_0$.
\end{proof}

\begin{definition}\label{direct.def.coherent-unitary-equivalences}
  Let $W\in\SegRev(\An)$.
  Define the anima $W_1^{\rmu}$ of \emph{unitary equivalences} in $W$ by the following pullback in $\An$:
  \begin{align}
    W_1^{\rmu}
    &:=
    (W_0\times W_0)
    \times_{c_{W_0\times W_0},\Map_{\An}(BC_2,W_0\times W_0),(\pr_1)^{hC_2}}
    (W_1^{\eqtag,C_2})^{hC_2}.
  \end{align}
  Here $c_{W_0\times W_0}$ is the constant-path morphism, and $\pr_1$ is the first projection of the pullback in \cref{direct.construction.equivalence-action}.
\end{definition}

\begin{lemma}\label{direct.lem.unitaries-as-hermitian-paths}
  Let $W\in\SegRev(\An)$ and write $\ac_{\flag}W=(i_W:A_W\to\Herm(\calC_W))$.
  The following statements hold.
  \begin{enumerate}
    \item There is a natural equivalence of anima
    \begin{align}
      W_1^{\rmu}
      &\simeq
      (W_0\times W_0)
      \times_{(i_W(\eta_W^{\flag})_0)^2,\Herm(\calC_W)^2,(\ev_0,\ev_1)}
      \Fun([1],\Herm(\calC_W)).
    \end{align}
    \item The equivalence in (1) is over $W_0\times W_0$:
    the two endpoint evaluations on $W_1^{\rmu}$ correspond to the projections to the two factors of $W_0\times W_0$.
  \end{enumerate}
\end{lemma}

\begin{proof}
  We prove (1) and (2) together.
  By \cref{target.lem.constant-groupoid-limit} and \cref{target.lem.homotopy-fixed-point-formalism} (6), (7), taking homotopy fixed points of the pullback in \cref{direct.construction.equivalence-action} gives
  \begin{align}
    (W_1^{\eqtag,C_2})^{hC_2}
    &\simeq
    \Map_{\An}(BC_2,W_0\times W_0)
    \times_{((\widetilde i_W)^2)^{hC_2},\Herm(\calC_W)^2,(\ev_{0},\ev_{1})}
    \Fun([1],\Herm(\calC_W)).
  \end{align}
  Let $c_{W_0}:W_0\to\Map_{\An}(BC_2,W_0)$ be the constant-path morphism.
  The adjunction defining $\widetilde i_W$ gives
  \begin{align}
    (\widetilde i_W)^{hC_2}c_{W_0}
    \simeq
    i_W(\eta_W^{\flag})_0:
    W_0\to\Herm(\calC_W).
  \end{align}
  Taking its square and pulling back along $c_{W_0\times W_0}$ as in \cref{direct.def.coherent-unitary-equivalences} gives (1).
  All comparisons are natural in $W$ and preserve the projection to $W_0\times W_0$.
\end{proof}

\begin{lemma}\label{direct.lem.unitary-forgetful-and-degeneracy}
  Let $W\in\SegRev(\An)$ and write $\ac_{\flag}W=(i_W:A_W\to\Herm(\calC_W))$.
  The following statements hold.
  \begin{enumerate}
    \item Evaluation of a homotopy fixed point at $*\in BC_2$ gives the forgetful morphism
    \begin{align}
      W_1^{\rmu}
      \to (\iota_{\rev}^{\myop,*}W)_1^{\eqtag}
      \to (\iota_{\rev}^{\myop,*}W)_1.
    \end{align}
    Under \cref{direct.lem.unitaries-as-hermitian-paths} (1), this morphism is induced by forgetting the Hermitian structure on the path.
    \item Taking homotopy fixed points of the equivariant identity arrows defined by $\widetilde i_W$, and using $c_{W_0}$, gives a morphism of anima 
    \begin{align}
      s_0^{\rmu}
      :W_0\to W_1^{\rmu}.
    \end{align}
    Under \cref{direct.lem.unitaries-as-hermitian-paths} (1), it sends $w\in W_0$ to the constant path at $i_W(\eta_W^{\flag})_0(w)$.
    \item Its composite with $W_1^{\rmu}\to W_1$ is the ordinary degeneracy $s_0:W_0\to W_1$.
  \end{enumerate}
\end{lemma}

\begin{proof}
  (1)
  Compose the second projection in \cref{direct.def.coherent-unitary-equivalences} with evaluation at $*\in BC_2$ and the equivalence of \cref{direct.lem.equivalence-action}.
  This is over $W_0\times W_0$, since $\ev_*c_{W_0\times W_0}\simeq\id_{W_0\times W_0}$.
  Naturality of evaluation identifies its path component in \cref{direct.lem.unitaries-as-hermitian-paths} with
  \begin{align}
    \Fun([1],\Herm(\calC_W))
    \xrightarrow{\Fun([1],\epsilon_{\calC_W})}
    \Fun([1],\iota_0(\calC_W)_*)
    \hookrightarrow\iota_0\Fun([1],(\calC_W)_*).
  \end{align}

  (2)
  The equivariant constant-path morphism gives a composite
  \begin{align}
    \triv(W_0)
    \xrightarrow{\widetilde i_W}
    \iota_0^{hC_2}\calC_W
    \to
    \Fun^{\conj}(\triv([1]),\iota_0^{hC_2}\calC_W).
  \end{align}
  Its endpoint morphism is $(\widetilde i_W)^2\triv(\Delta_{W_0})$.
  Thus \cref{direct.construction.equivalence-action} gives a morphism $\triv(W_0)\to W_1^{\eqtag,C_2}$ over $\triv(\Delta_{W_0})$.
  Taking homotopy fixed points and precomposing with $c_{W_0}$ gives a morphism over $c_{W_0\times W_0}\Delta_{W_0}$, so \cref{direct.def.coherent-unitary-equivalences} supplies $s_0^{\rmu}$.
  Naturality of constant paths and the adjunction defining $\widetilde i_W$ identify its path component with the constant path at $i_W(\eta_W^{\flag})_0(w)$.

  (3)
  By (1) and (2), forgetting the Hermitian structure on $s_0^{\rmu}(w)$ gives the constant path at $\epsilon_{\calC_W}i_W(\eta_W^{\flag})_0(w)$.
  The comparison in \cref{direct.lem.equivalence-action} is induced by the simplicial unit, which is an equivalence and commutes with degeneracies.
  Hence the composite is $s_0$.
\end{proof}

\begin{lemma}\label{direct.lem.walking-unitary-evaluation}
  Let $X=(i:A\to\Herm(\calC))\in\HermFlag_{(\infty,1)}$.
  The following statements hold.
  \begin{enumerate}
    \item There is a natural equivalence of anima
    \begin{align}
      \Map_{\Fun(\prism_{\rev}^{\myop},\An)}(E_{\rev},\rmN_{\flag}X)
      \simeq
      (A\times A)
      \times_{i^2,\Herm(\calC)^2,(\ev_0,\ev_1)}
      \Fun([1],\Herm(\calC)).
    \end{align}
    \item Under the equivalence in (1), evaluation at the vertices $0,1$ corresponds to the two endpoints in $A$.
    In particular, restriction along $e_{\rev}$ corresponds to the first endpoint projection.
    \item Under the equivalence in (1), evaluation of the arrow $(0,1)$ in $(\rmN_{\flag}X)_1$ corresponds to forgetting the Hermitian structure on the path.
    \item Under the equivalence in (1), precomposition with $E_{\rev}\to\yo_{\rev}[0]$ sends $a\in A$ to the constant path at $i(a)$.
  \end{enumerate}
\end{lemma}

\begin{proof}
  (1), (2)
  Since $\mathfrak u$ is flagged, \cref{direct.thm.flagged-equivalence} (2), \cref{direct.lem.walking-unitary}, and the proof of \cref{direct.lem.e-and-u} (2) give
  \begin{align}
    \Map_{\Fun(\prism_{\rev}^{\myop},\An)}(E_{\rev},\rmN_{\flag}X)
    &\simeq
    \Map_{\HermFlag_{(\infty,1)}}(\mathfrak u,X)\\
    &\simeq
    (A\times A)
    \times_{i^2,\Herm(\calC)^2,(\ev_0,\ev_1)}
    \Map_{\An}(I,\Herm(\calC)).
  \end{align}
  The functor $[1]\to I$ selecting $0\to1$ is a groupoid completion.
  Since $\Herm(\calC)$ is an anima, restriction identifies the last path factor with $\Fun([1],\Herm(\calC))$, compatibly with both endpoints.
  All comparisons are natural in $X$, and restriction along $e_{\rev}$ corresponds to the first endpoint, since $e_{\rev}$ selects $0$.

  (3)
  The morphism $\yo_{\rev}[1]\to E_{\rev}$ determined by $(0,1)$ corresponds to $[1]_{\flag}\to\mathfrak u$, sending the vertices to $0,1$ and the generating arrow to $0\to1$ in $I$.
  Under the preceding equivalences, precomposition is therefore induced by
  \begin{align}
    \Fun([1],\Herm(\calC))
    \xrightarrow{\Fun([1],\epsilon_{\calC})}
    \Fun([1],\iota_0\calC_*)
    \hookrightarrow\iota_0\Fun([1],\calC_*).
  \end{align}
  By \cref{flagged-nerve.lem.value-of-flagged-nerve}, its pullback along the specified endpoints is evaluation of $(0,1)$ in $(\rmN_{\flag}X)_1$.

  (4)
  The morphism $E_{\rev}\to\yo_{\rev}[0]$ corresponds to $\mathfrak u\to\mathfrak e$ collapsing both vertices to the selected object.
  Precomposition sends $a\in A$ to the constant path at $i(a)$.
\end{proof}

\begin{proposition}\label{direct.thm.coherent-unitary-representability}
  Let $W\in\SegRev(\An)$.
  The following statements hold.
  \begin{enumerate}
    \item There is a natural equivalence of anima
    \begin{align}
      \Map_{\Fun(\prism_{\rev}^{\myop},\An)}(E_{\rev},W)
      \simeq W_1^{\rmu}.
    \end{align}
    \item The equivalence in (1) is compatible with both endpoint evaluations.
    In particular, restriction along $e_{\rev}$ corresponds to the first endpoint projection $W_1^{\rmu}\to W_0$.
    \item Under the equivalence in (1), evaluation of the arrow $(0,1)$ in $W_1$ corresponds to the forgetful morphism $W_1^{\rmu}\to W_1$.
    \item Under the equivalence in (1), precomposition with $E_{\rev}\to\yo_{\rev}[0]$ corresponds to $s_0^{\rmu}:W_0\to W_1^{\rmu}$.
  \end{enumerate}
\end{proposition}

\begin{proof}
  (1)
  Write $\ac_{\flag}W=(i_W:A_W\to\Herm(\calC_W))$.
  By \cref{direct.thm.unit} (1), both $\eta_W^{\flag}:W\to\rmN_{\flag}\ac_{\flag}W$ and $(\eta_W^{\flag})_0:W_0\to A_W$ are equivalences.
  Applying \cref{direct.lem.walking-unitary-evaluation} (1), pulling back along $(\eta_W^{\flag})_0\times(\eta_W^{\flag})_0$, and using \cref{direct.lem.unitaries-as-hermitian-paths} (1), we obtain
  \begin{align}
    \Map_{\Fun(\prism_{\rev}^{\myop},\An)}(E_{\rev},W)
    &\simeq
    \Map_{\Fun(\prism_{\rev}^{\myop},\An)}(E_{\rev},\rmN_{\flag}\ac_{\flag}W)\\*
    &\simeq W_1^{\rmu}.
  \end{align}

  (2)
  Part (2) of \cref{direct.lem.unitaries-as-hermitian-paths,direct.lem.walking-unitary-evaluation} and naturality of the unit identify both endpoint evaluations.
  In particular, restriction along $e_{\rev}$ corresponds to the first endpoint projection.

  (3)
  By \cref{direct.lem.walking-unitary-evaluation} (3) and \cref{direct.lem.unitary-forgetful-and-degeneracy} (1), both morphisms are induced on paths by $\Fun([1],\epsilon_{\calC_W})$.
  Naturality of the unit in degree one gives the asserted compatibility.

  (4)
  By \cref{direct.lem.walking-unitary-evaluation} (4) and \cref{direct.lem.unitary-forgetful-and-degeneracy} (2), both morphisms are induced by constant paths.
  Naturality of the unit in degree zero gives the asserted compatibility.
\end{proof}

\begin{corollary}\label{direct.cor.coherent-unitary-completeness}
  Let $W\in\SegRev(\An)$ and let $X=(i:A\to\Herm(\calC))\in\HermFlag_{(\infty,1)}$.
  The following statements hold.
  \begin{enumerate}
    \item $W$ is unitarily complete if and only if $s_0^{\rmu}:W_0\to W_1^{\rmu}$ is an equivalence.
    \item There is a natural equivalence of anima over $A\times A$:
    \begin{align}
      (\rmN_{\flag}X)_1^{\rmu}
      \simeq A\times_{\Herm(\calC)}A.
    \end{align}
    \item The equivalence in (2) identifies $s_0^{\rmu}$ with the diagonal and the first endpoint evaluation with the first projection.
    \item $\rmN_{\flag}X$ is unitarily complete if and only if $i$ is fully faithful.
  \end{enumerate}
\end{corollary}

\begin{proof}
  (1)
  By \cref{direct.thm.coherent-unitary-representability} (1), (2), (4), unitary completeness is equivalent to the first endpoint projection being an equivalence, and $\ev_0s_0^{\rmu}\simeq\id_{W_0}$.
  The two-out-of-three property proves (1).

  (2), (3)
  By \cref{direct.thm.coherent-unitary-representability} (1), (2), (4) and \cref{direct.lem.walking-unitary-evaluation} (1), (2), (4), there are natural equivalences over $A\times A$:
  \begin{align}
    (\rmN_{\flag}X)_1^{\rmu}
    &\simeq
    (A\times A)
    \times_{i^2,\Herm(\calC)^2,(\ev_0,\ev_1)}
    \Fun([1],\Herm(\calC))\\
    &\simeq A\times_{\Herm(\calC)}A.
  \end{align}
  The last equivalence uses the constant-path equivalence $\Herm(\calC)\xrightarrow{\simeq}\Fun([1],\Herm(\calC))$ over $\Herm(\calC)^2$.
  The endpoint evaluations correspond to the projections to $A$, and $s_0^{\rmu}$ corresponds to the diagonal.

  (4)
  By (1) and (3), $\rmN_{\flag}X$ is unitarily complete if and only if $\Delta:A\to A\times_{\Herm(\calC)}A$ is an equivalence.
  This is equivalent to full faithfulness of $i$.
\end{proof}

\subsection{The unitary Dwyer--Kan criterion}

We characterize equivalences after unitary completion by full faithfulness and essential surjectivity up to unitary equivalences (\cref{presentation.thm.unitary-dk}).

\begin{definition}\label{presentation.def.unitary-dk}
  Let $f:W\to V$ be a morphism in $\SegRev(\An)$.
  We will say that $f$ is
  \begin{enumerate}
    \item \emph{fully faithful} if $\Map_W(x,y)\to\Map_V(f_0x,f_0y)$ is an equivalence for every $x,y\in W_0$;
    \item \emph{unitarily essentially surjective} if $W_0\times_{f_0,V_0,\ev_{0}}V_1^{\rmu} \xrightarrow{\ev_{1} \operatorname{pr}_2}V_0$ is essentially surjective.
  \end{enumerate}
\end{definition}

\begin{lemma}\label{presentation.lem.flagged-dk}
  Let $X=(i:A\to\Herm(\calC)), Y=(j:B\to\Herm(\calD)) \in \FlDagCat_{(\infty,1)}$, and let $(F,f):X\to Y$ be a morphism in $\FlDagCat_{(\infty,1)}$.
  Then $L_{\co\flag}|_{\flag}(F,f)$ is an equivalence if and only if both of the following conditions hold:
  \begin{enumerate}
    \item $F_*:\calC_*\to\calD_*$ is fully faithful;
    \item the projection $A\times_{\Herm(F)i,\Herm(\calD),j}B\to B$ is essentially surjective.
  \end{enumerate}
\end{lemma}

\begin{proof}
  Functorial image factorization supplies a morphism $\bar f : \image(i)\to\image(j)$ over $\Herm(F)$.
  By \cref{flag.thm.univalentization}, the underlying $(\infty,1)$-category component of $L_{\co\flag}|_{\flag}(F,f)$ is $F_*$ and its flag component is $\bar f$.
  By \cref{target.lem.hermitian-flag-equivalences}, $L_{\co\flag}|_{\flag}(F,f)$ is an equivalence if and only if $F_*$ and $\bar f$ are equivalences.

  ($\Rightarrow$)
  Suppose first that $L_{\co\flag}|_{\flag}(F,f)$ is an equivalence.
  Then $F_*$ and $\bar f$ are equivalences, so condition (1) holds.
  Write the image factorizations as $A \xtwoheadrightarrow{q_i} \image(i) \xhookrightarrow{\ell_i} \Herm(\calC)$ and $B \xtwoheadrightarrow{q_j} \image(j) \xhookrightarrow{\ell_j} \Herm(\calD)$.
  They satisfy $\ell_iq_i\simeq i$, $\ell_jq_j\simeq j$, and $\ell_j\bar f\simeq\Herm(F)\ell_i$.
  For $b\in B$, essential surjectivity of $q_i$ and the equivalence $\bar f$ give an $a\in A$ and a path $\bar f(q_i(a))\simeq q_j(b)$ in $\image(j)$.
  Applying $\ell_j$ gives a path in $\Herm(\calD)$:
  \begin{align}
    \Herm(F)i(a)
    \simeq \Herm(F)\ell_i(q_i(a)) 
    \simeq \ell_j\bar f(q_i(a)) 
    \simeq \ell_j(q_j(b))
    \simeq j(b).
  \end{align}
  This gives a point in the fiber of the projection in condition~(2) over $b$.
  Hence condition (2) holds.

  ($\Leftarrow$)
  Assume (1) and (2).
  Every object of $\calD_*$ is equivalent to $\epsilon_{\calD}j(b)$ for some $b\in B$.
  Condition (2), followed by $\epsilon_{\calD}$, gives an $x\in A$ and an equivalence $F_*(\epsilon_{\calC}i(x)) \simeq \epsilon_{\calD}j(b)$.
  Thus $F_*$ is essentially surjective, and (1) makes it an equivalence.
  Consequently $F$ and $\Herm(F)$ are equivalences.

  The morphism $\bar f$ is the restriction of this equivalence of Hermitian anima to two full subanima.
  It is therefore fully faithful.
  Every component of $\image(j)$ contains $j(b)$ for some $b\in B$, and (2) places it in the image of $\bar f$.
  Hence $\bar f$ is also essentially surjective and thus an equivalence.
  The criterion above now shows that $L_{\co\flag}|_{\flag}(F,f)$ is an equivalence.
\end{proof}

\begin{proposition}[The unitary Dwyer--Kan criterion]\label{presentation.thm.unitary-dk}
  Let $f : W \to V$ be a morphism in $\Seg_{\rev}(\An)$.
  Then the following conditions are equivalent:
  \begin{enumerate}
    \item $L_{\rmu}f$ is an equivalence.
    \item $f$ is fully faithful and unitarily essentially surjective.
  \end{enumerate}
\end{proposition}

\begin{proof}
  By \cref{direct.thm.unit} (1), (2), $\ac_{\flag}W$ and $\ac_{\flag}V$ are flagged, and the unit maps $\eta^{\flag}_W$ and $\eta^{\flag}_V$ are equivalences.
  Naturality of the unit identifies $f$ with $\rmN_{\flag}\ac_{\flag}f$ under these equivalences.
  By \cref{flag.cor.unitary-completion}, $L_{\rmu}f$ is an equivalence if and only if $L_{\co\flag}|_{\flag}(\ac_{\flag}f)$ is an equivalence.
  We compare the two conditions of \cref{presentation.lem.flagged-dk}, applied to $\ac_{\flag}f$, with those in condition (2).
  Write $\ac_{\flag}W=(i_W:A_W\to\Herm(\calC_W))$ and $\ac_{\flag}V = (i_V:A_V\to\Herm(\calC_V))$.

  First, the degree-one pullback of \cref{flagged-nerve.lem.value-of-flagged-nerve} and the unit give, naturally in $x,y\in W_0$,
  \begin{align}
    \Map_W(x,y)
    \simeq
    \Map_{(\calC_W)_*}
    (\epsilon_{\calC_W}i_W(\eta^{\flag}_W)_0(x),
    \epsilon_{\calC_W}i_W(\eta^{\flag}_W)_0(y)).
  \end{align}
  Naturality identifies the map induced by $f$ with the map induced by the underlying functor of $\ac_{\flag}f$.
  Since $\ac_{\flag}W$ is flagged and $(\eta^{\flag}_W)_0:W_0\to A_W$ is an equivalence, the composite $\epsilon_{\calC_W}i_W(\eta^{\flag}_W)_0$ is essentially surjective.
  Transport along equivalences therefore shows that $f$ is fully faithful if and only if the underlying functor of $\ac_{\flag}f$ is fully faithful.
  This is condition (1) of \cref{presentation.lem.flagged-dk}.

  Next, the degree-zero case of \cref{flagged-nerve.lem.value-of-flagged-nerve} identifies the flag component of $\ac_{\flag}f$ with $(\rmN_{\flag}\ac_{\flag}f)_0:A_W\to A_V$.
  Naturality of the unit and \cref{direct.cor.coherent-unitary-completeness} (2), (3) give
  \begin{align}
    W_0\times_{f_0,V_0,\ev_0}V_1^{\rmu}
    &\simeq
    A_W\times_{(\rmN_{\flag}\ac_{\flag}f)_0,A_V,\ev_0}
    (\rmN_{\flag}\ac_{\flag}V)_1^{\rmu}\\
    &\simeq
    A_W\times_{i_V(\rmN_{\flag}\ac_{\flag}f)_0,\Herm(\calC_V),i_V}A_V.
  \end{align}
  The first equivalence covers $(\eta^{\flag}_V)_0:V_0\xrightarrow{\simeq}A_V$, and the second is over $A_V$.
  Since $\ac_{\flag}f$ is a morphism of Hermitian flags, its compatibility with the flag maps identifies the last projection to $A_V$ with the projection in \cref{presentation.lem.flagged-dk} (2).
  Hence $f$ is unitarily essentially surjective if and only if condition (2) of that lemma holds for $\ac_{\flag}f$.

  By \cref{presentation.lem.flagged-dk}, these two comparisons identify condition (2) of the proposition with $L_{\co\flag}|_{\flag}(\ac_{\flag}f)$ being an equivalence.
\end{proof}

\section*{References}
\printbibliography[heading=none]
\end{document}